\documentclass[12pt,leqno,draft]{article}
\usepackage{amsfonts}
\usepackage{amsmath, amsthm, amsfonts, amssymb, color}
\usepackage{mathrsfs}
\usepackage{color}
\newtheorem{thm}{Theorem}[section]
\newtheorem{cor}[thm]{Corollary}
\newtheorem{lem}[thm]{Lemma}

\newtheorem{rem}[thm]{Remark}
\theoremstyle{definition}

\newcommand{\scr}[1]{\mathscr #1}
\definecolor{wco}{rgb}{0.5,0.2,0.3}

\numberwithin{equation}{section} \theoremstyle{remark}

\newcommand{\ua}{\uparrow}

\title{{\bf   Well-posedness and Regularity for Distribution Dependent SPDEs with Singular Drifts  }\footnote{Supported in
 part by  NNSFC (11801406,11501286,
11790272).} }
\author{
{\bf     Xing Huang$^{a)}$, Yulin Song$^{b)}$   }\\
\footnotesize{  a)Center for Applied Mathematics, Tianjin University, Tianjin 300072, China}\\
\footnotesize{  xinghuang@tju.edu.cn}\\
\footnotesize{$^{b)}$Department of Mathematics, Nanjing University, Nanjing, 210093, China}\\
\footnotesize{ songyl@amss.ac.cn}}
\begin{document}
\allowdisplaybreaks
\def\R{\mathbb R}  \def\ff{\frac} \def\ss{\sqrt} \def\B{\mathbf
B}
\def\N{\mathbb N} \def\kk{\kappa} \def\m{{\bf m}}
\def\ee{\varepsilon}\def\ddd{D^*}
\def\dd{\delta} \def\DD{\Delta} \def\vv{\varepsilon} \def\rr{\rho}
\def\<{\langle} \def\>{\rangle} \def\GG{\Gamma} \def\gg{\gamma}
  \def\nn{\nabla} \def\pp{\partial} \def\E{\mathbb E}
\def\d{\text{\rm{d}}} \def\bb{\beta} \def\aa{\alpha} \def\D{\scr D}
  \def\si{\sigma} \def\ess{\text{\rm{ess}}}
\def\beg{\begin} \def\beq{\begin{equation}}  \def\F{\scr F}
\def\Ric{\text{\rm{Ric}}} \def\Hess{\text{\rm{Hess}}}
\def\e{\text{\rm{e}}} \def\ua{\underline a} \def\OO{\Omega}  \def\oo{\omega}
 \def\tt{\tilde} \def\Ric{\text{\rm{Ric}}}
\def\cut{\text{\rm{cut}}} \def\P{\mathbb P} \def\ifn{I_n(f^{\bigotimes n})}
\def\C{\scr C}      \def\aaa{\mathbf{r}}     \def\r{r}
\def\gap{\text{\rm{gap}}} \def\prr{\pi_{{\bf m},\varrho}}  \def\r{\mathbf r}
\def\Z{\mathbb Z} \def\vrr{\varrho} \def\ll{\lambda}
\def\L{\scr L}\def\Tt{\tt} \def\TT{\tt}\def\II{\mathbb I}
\def\i{{\rm in}}\def\Sect{{\rm Sect}}  \def\H{\mathbb H}
\def\M{\scr M}\def\Q{\mathbb Q} \def\texto{\text{o}} \def\LL{\Lambda}
\def\Rank{{\rm Rank}} \def\B{\scr B} \def\i{{\rm i}} \def\HR{\hat{\R}^d}
\def\to{\rightarrow}\def\l{\ell}\def\iint{\int}
\def\EE{\scr E}\def\no{\nonumber}
\def\A{\scr A}\def\V{\mathbb V}\def\osc{{\rm osc}}
\def\BB{\scr B}\def\Ent{{\rm Ent}}
\def\U{\scr U}
\def\R{\mathbb R}  \def\ff{\frac} \def\ss{\sqrt} \def\B{\mathbf
B}
\def\N{\mathbb N} \def\kk{\kappa} \def\m{{\bf m}}
\def\ee{\varepsilon}\def\ddd{D^*}
\def\dd{\delta} \def\DD{\Delta} \def\vv{\varepsilon} \def\rr{\rho}
\def\<{\langle} \def\>{\rangle} \def\GG{\Gamma} \def\gg{\gamma}
  \def\nn{\nabla} \def\pp{\partial} \def\E{\mathbb E}
\def\d{\text{\rm{d}}} \def\bb{\beta} \def\aa{\alpha} \def\D{\scr D}
  \def\si{\sigma} \def\ess{\text{\rm{ess}}}
\def\beg{\begin} \def\beq{\begin{equation}}  \def\F{\scr F}
\def\Ric{\text{\rm{Ric}}} \def\Hess{\text{\rm{Hess}}}
\def\e{\text{\rm{e}}} \def\ua{\underline a} \def\OO{\Omega}  \def\oo{\omega}
 \def\tt{\tilde} \def\Ric{\text{\rm{Ric}}}
\def\cut{\text{\rm{cut}}} \def\P{\mathbb P} \def\ifn{I_n(f^{\bigotimes n})}
\def\C{\scr C}   \def\G{\scr G}   \def\aaa{\mathbf{r}}     \def\r{r}
\def\gap{\text{\rm{gap}}} \def\prr{\pi_{{\bf m},\varrho}}  \def\r{\mathbf r}
\def\Z{\mathbb Z} \def\vrr{\varrho} \def\ll{\lambda}
\def\L{\scr L}\def\Tt{\tt} \def\TT{\tt}\def\II{\mathbb I}
\def\i{{\rm in}}\def\Sect{{\rm Sect}}  \def\H{\mathbb H}
\def\M{\scr M}\def\Q{\mathbb Q} \def\texto{\text{o}} \def\LL{\Lambda}
\def\Rank{{\rm Rank}} \def\B{\scr B} \def\i{{\rm i}} \def\HR{\hat{\R}^d}
\def\to{\rightarrow}\def\l{\ell}\def\iint{\int}
\def\EE{\scr E}\def\no{\nonumber}
\def\A{\scr A}\def\V{\mathbb V}\def\osc{{\rm osc}}
\def\BB{\scr B}\def\Ent{{\rm Ent}}\def\3{\triangle}
\def\U{\scr U}\def\8{\infty}\def\1{\lesssim}
\def\R{\mathbb R}  \def\ff{\frac} \def\ss{\sqrt} \def\B{\mathbf
B} \def\W{\mathbb W}
\def\N{\mathbb N} \def\kk{\kappa} \def\m{{\bf m}}
\def\ee{\varepsilon}\def\ddd{D^*}
\def\dd{\delta} \def\DD{\Delta} \def\vv{\varepsilon} \def\rr{\rho}
\def\<{\langle} \def\>{\rangle} \def\GG{\Gamma} \def\gg{\gamma}
  \def\nn{\nabla} \def\pp{\partial} \def\E{\mathbb E}
\def\d{\text{\rm{d}}} \def\bb{\beta} \def\aa{\alpha} \def\D{\scr D}
  \def\si{\sigma} \def\ess{\text{\rm{ess}}}
\def\beg{\begin} \def\beq{\begin{equation}}  \def\F{\scr F}
\def\Ric{\text{\rm{Ric}}} \def\Hess{\text{\rm{Hess}}}
\def\e{\text{\rm{e}}} \def\ua{\underline a} \def\OO{\Omega}  \def\oo{\omega}
 \def\tt{\tilde} \def\Ric{\text{\rm{Ric}}}
\def\cut{\text{\rm{cut}}} \def\P{\mathbb P} \def\ifn{I_n(f^{\bigotimes n})}
\def\C{\scr C}      \def\aaa{\mathbf{r}}     \def\r{r}
\def\gap{\text{\rm{gap}}} \def\prr{\pi_{{\bf m},\varrho}}  \def\r{\mathbf r}
\def\Z{\mathbb Z} \def\vrr{\varrho} \def\ll{\lambda}
\def\L{\scr L}\def\Tt{\tt} \def\TT{\tt}\def\II{\mathbb I}
\def\i{{\rm in}}\def\Sect{{\rm Sect}}  \def\H{\mathbb H}
\def\M{\scr M}\def\Q{\mathbb Q} \def\texto{\text{o}} \def\LL{\Lambda}
\def\Rank{{\rm Rank}} \def\B{\scr B} \def\i{{\rm i}} \def\HR{\hat{\R}^d}
\def\to{\rightarrow}\def\l{\ell}\def\iint{\int}
\def\EE{\scr E}\def\Cut{{\rm Cut}}
\def\A{\scr A} \def\Lip{{\rm Lip}}
\def\BB{\scr B}\def\Ent{{\rm Ent}}\def\L{\scr L}
\def\R{\mathbb R}  \def\ff{\frac} \def\ss{\sqrt} \def\B{\mathbf
B}
\def\N{\mathbb N} \def\kk{\kappa} \def\m{{\bf m}}
\def\dd{\delta} \def\DD{\Delta} \def\vv{\varepsilon} \def\rr{\rho}
\def\<{\langle} \def\>{\rangle} \def\GG{\Gamma} \def\gg{\gamma}
  \def\nn{\nabla} \def\pp{\partial} \def\E{\mathbb E}
\def\d{\text{\rm{d}}} \def\bb{\beta} \def\aa{\alpha} \def\D{\scr D}
  \def\si{\sigma} \def\ess{\text{\rm{ess}}}
\def\beg{\begin} \def\beq{\begin{equation}}  \def\F{\scr F}
\def\Ric{\text{\rm{Ric}}} \def\Hess{\text{\rm{Hess}}}
\def\e{\text{\rm{e}}} \def\ua{\underline a} \def\OO{\Omega}  \def\oo{\omega}
 \def\tt{\tilde} \def\Ric{\text{\rm{Ric}}}
\def\cut{\text{\rm{cut}}} \def\P{\mathbb P} \def\ifn{I_n(f^{\bigotimes n})}
\def\C{\scr C}      \def\aaa{\mathbf{r}}     \def\r{r}
\def\gap{\text{\rm{gap}}} \def\prr{\pi_{{\bf m},\varrho}}  \def\r{\mathbf r}
\def\Z{\mathbb Z} \def\vrr{\varrho} \def\ll{\lambda}
\def\L{\scr L}\def\Tt{\tt} \def\TT{\tt}\def\II{\mathbb I}
\def\i{{\rm in}}\def\Sect{{\rm Sect}}  \def\H{\mathbb H}
\def\M{\scr M}\def\Q{\mathbb Q} \def\texto{\text{o}} \def\LL{\Lambda}
\def\Rank{{\rm Rank}} \def\B{\scr B} \def\i{{\rm i}} \def\HR{\hat{\R}^d}
\def\to{\rightarrow}\def\l{\ell}
\def\8{\infty}\def\I{1}\def\U{\scr U}
\maketitle

\begin{abstract} In this paper, the distribution dependent stochastic differential equation in a separable Hilbert space with a Dini continuous drift is investigated. The existence and uniqueness of weak and strong solutions are obtained. Moreover, some regularity results as well as gradient estimates and log-Harnack inequality are derived for the associated semigroup. In addition, dimensional free Harnack inequality with power and shift Harnack inequality are also proved when the noise is additive.
All of the results extend the ones in the distribution independent situation.
\end{abstract} \noindent
 \vskip 2cm

\section{Introduction}
The distribution dependent stochastic differential equations (SDEs for short), also named McKean-Vlasov SDEs due to pioneering work \cite{M,V}, can be described as the weak limit of
$N$-particle interaction systems formed by $N$ equations forced by independent
Brownian motions. The subject has been extensively explored and it is still under investigation (see \cite{BR1,BR2,CMF,CF,FS,HRW,M,V,FYW1} and references within). When the drifts are singular,
there are a great number of results on the well-posedness , for instance,
\cite{BB0,BB,BBP,CR,HW,MV,RZ} and references therein. In \cite{BB0,BB,BBP}, the existence of weak solutions in the additive noise case is shown by Girsanov's transform together with Schauder's fixed point theorem. However, this method does not work when the diffusion coefficients depend on distribution. The results in \cite{CR} are extended by the first author and his coauthor in \cite{HW}, where the diffusion term is allowed to be distribution dependent. The pathwise uniqueness is proved by utilizing Zvonkin's transform \cite{AZ} in \cite{HW,MV,RZ}, see references therein for distribution independent SDEs. The main idea of Zvonkin's transform is to remove the singular drifts, and it mainly depends on the regularity of a backward Kolmogrov equation with singular coefficients. In the infinite dimensional and distribution independent case, the author in \cite{W} investigates the existence and uniqueness of solutions and log-Harnack inequality for semi-linear stochastic partial differential equations (SPDEs) with Dini continuous drifts by Zvonkin's transform.

The present paper attempts to extend the results in \cite{W} to the distribution dependent case. Meanwhile, dimension-free Harnack inequality with power and shift Harnack inequality are also considered in special situations. In order to obtain the existence of weak solutions under a weak condition, the compactness method [11, chapter 8] as well as Skorohod representation and martingale representation theorem will be employed. It is crucial to construct a family of compact operators to  deal with the stochastic convolution. Moreover, Zvonkin's transform combined with fixed point theorem can be used to investigate the strong well-posedness.

Using the method of coupling by change of measure, the dimension-free Harnack inequality, log-Harnack inequality and shift-Harnack inequality, introduced by F.-Y Wang in \cite{W97}, \cite{RW} and \cite{W14a} respectively, have been established and applied to various SDEs and SPDEs driven by Gaussian noises, see \cite{L,RW10,RW,W14a,Wbook,W07,WZ} and references therein. Different from the finite dimensional case \cite[Theorem 2.5]{HW}, due to the existence of a non-Lipschitzian term $A u$ after Zvonkin's transform  in Lemma \ref{L3.2} below, the coupling by change of measure, for instance in \cite[Chapter 3]{Wbook}, does not work even in the distribution independent case with multiplicative noise. To overcome this difficulty, \cite{W} adopted the gradient-gradient estimate for Markovian semigroups to derive the log-Harnack inequality according to \cite[Chapter 1]{Wbook}. However, this method is unavailable in the distribution dependent case since the solution is not a Markov process. Fortunately, we may employ the existed log-Harnack inequality in \cite{W} and Girsanov's transform to obtain the desired log-Harnack inequality. The main idea is to derive the estimate of the relative entropy between two solutions with different initial distributions. To this end, we rewrite one of the two solutions by Girsanov's transform to be a new one with  the same coefficients with another one, and then the log-Harnack inequality in \cite{W} can be used.
It seems that this method is an effective way to deal with the distribution dependent SDEs and SPDEs.
As for the  Harnack inequality with power and shift Harnack inequality, we adopt coupling by change of measure instead of Zvonkin's transform in the additive noise case.

Let $(\mathbb{H},\langle,\rangle,|\cdot|)$ and $(\mathbb{\bar{H}},\langle,\rangle_{\mathbb{\bar{H}}},|\cdot|_{\mathbb{\bar{H}}})$ be two separable Hilbert spaces, and $W=(W_t)_{t\geq 0}$ be a cylindrical Brownian motion on $\mathbb{\bar{H}}$ with respect to a complete filtered  probability space $(\OO, \F, \{\F_{t}\}_{t\ge 0}, \P)$. More precisely, $W_t=\sum_{n=1}^{\infty}{B^{n}_t\bar{e}_{n}}$ for a sequence of independent one dimensional Brownian motions $\{B^{n}_t\}_{n\geq 1}$ with respect to $(\OO, \F,
\{\F_{t}\}_{t\ge 0}, \P)$ and an orthonormal basis $\{\bar{e}_{n}\}_{n\geq 1}$ on $\mathbb{\bar{H}}$.

Let $\mathscr{P}$ be the
collection of all probability measures on $\H$ equipped with the weak topology. For $\mu\in\scr P$, if
$\mu(|\cdot|^p):=\int_{\H}|x|^p\mu(\d x)<\8$ for some $p\geq1$, we write
$\mu\in\mathscr{P}_p$.  For $p\geq 1$ and $\mu,\bar{\mu}\in\mathscr{P}_p$,
 the $\mathbb{W}_p$-Wasserstein distance between $\mu$ and
$\bar \mu$ is defined by
\begin{equation*}
\mathbb{W}_p(\mu,\bar{\mu})=\inf_{\pi\in\mathcal
{C}(\mu,\bar{\mu})}\Big(\int_{\H\times\H}|x-y|^p\pi(\d x,\d y)\Big)^{\ff{1}{
p}},
\end{equation*}
where $\mathcal {C}(\mu,\bar{\mu})$ stands for the set of all couplings of
$\mu$ and $\bar{\mu}$. For a random variable $\xi,$ its law is
written by $\mathscr{L}_\xi$, and write $\mathscr{L}_\xi|_{\P}$ as the distribution of $\xi$ under $\P$.

Consider the following semi-linear distribution dependent SPDEs on $\mathbb{H}$:
\beq\label{E1} \d X_t= \{A X_t+b_t(X_t, \L_{X_t})\}\d t+Q_t(X_t,\L_{X_t})\d W_t,
\end{equation}
where $(A,\D(A))$ is a negative definite self-adjoint operator on $\mathbb{H}$,
$b: [0,\infty)\times \mathbb{H}\times \scr P\to \mathbb{H}$ is measurable and locally bounded (i.e. bounded on bounded sets), and $Q: [0,\infty)\times \mathbb{H}\times \scr P\to \L(\mathbb{\bar{H}}; \mathbb{H})$ is measurable, where $\L(\mathbb{\bar{H}};\mathbb{H})$ is the space of bounded linear operators from $\mathbb{\bar{H}}$ to $\mathbb{H}$.
Let $\|\cdot\|$ and $\|\cdot\|_{{\rm{HS}}}$ denote the operator norm and the Hilbert-Schmidt norm respectively. %and let $\L_{HS}(\mathbb{\bar{H}}; \mathbb{H})$ be the space of all Hilbert-Schmidt operators from $\mathbb{\bar{H}}$ to $\mathbb{H}$.

To characterize the singularity of $b$ with respect to the second variable, set
\beg{equation*}\beg{split}
\D= \Big\{\phi: [0,+\infty)\to [0,+\infty)| \phi^{2} \text{ is concave and }\phi \text{ is increasing with } \int_0^1{\frac{\phi(s)}{s}\d s}<\infty\Big\}.
\end{split}\end{equation*}
Throughout this paper, we assume that there exists an increasing function $K:(0,\infty)\to (0,\infty)$ such that $A$, $b$ and $Q$ satisfy the following conditions.
\beg{enumerate}
\item[{\bf (a1)}]
For some $\varepsilon \in(0,1)$, $(-A)^{\varepsilon-1}$ is of trace class.
That is, $\sum_{n=1}^{\infty}{\lambda_{n}^{\varepsilon-1}}<\infty$ for $0< \lambda_{1}\leq \lambda_{2}\leq\cdots$ being all eigenvalues of $-A$ counting multiplicities with $-A e_i=\lambda_i e_i,i\geq 1$ for an orthonormal basis $\{e_i\}_{i\geq 1}$ of $\H$.

\item[{\bf (a2)}]
The operator $Q:[0,\infty)\times \mathbb{H}\times\scr P\rightarrow\L(\mathbb{\bar{H}}; \mathbb{H}))$ is continuous and for each $t\geq0$ and $\mu\in\scr P,$ and $Q_t(\cdot,\mu)$ is in $C^{2}(\mathbb{H};\L(\mathbb{\bar{H}};\mathbb{H}))$ such that
\beg{equation*}
\beg{split}\label{re0}
\sup_{(t,x,\mu)\in[0,T]\times \mathbb{H}\times\scr P}
\left(\| Q_t(x,\mu)\|+\|\nabla Q_t(x,\mu)\|+\|\nabla^{2} Q_t(x,\mu)\|\right)\leq K(T),\ \ T>0,
\end{split}
\end{equation*}
here $\nabla$ and $\nabla^2$ stand for the first and second ordered gradient operator with respect to the space component respectively.
Meanwhile, $(Q_t Q_t^{\ast})(x,\mu)$ is invertible for each $(t,x,\mu)\in[0,\infty)\times \mathbb{H}\times\scr P$ with
\beg{equation*}
\beg{split}\label{re1}
\sup_{(t,x,\mu)\in[0,T]\times \mathbb{H}\times\scr P}
\|(Q_t Q_t^{\ast})(x,\mu)^{-1}\|\leq K(T),\ \ T>0.
\end{split}
\end{equation*}
Moreover, for any $x\in\mathbb{H}$, $t\geq 0$ and $\mu\in\scr P_2$, it holds
\beq\label{1.3}
\lim_{n\to \infty}\|Q_t(x,\mu)-Q_t(\pi_{n}x,\mu)\|_{{\rm{HS}}}^{2}=0,
\end{equation}
where $\pi_n$ is the orthonormal projection from $\mathbb{H}$ to span$\{e_1, e_2, \cdots, e_n\}$.
In addition, for any $T>0$, it holds
\begin{align}\label{red}
\sup_{(t,x)\in[0,T]\times \mathbb{H}}\|Q_t(x,\mu)-Q_t(x,\nu)\|^2_{{\rm{HS}}}
\leq K(T)\W_2(\mu,\nu)^2,\ \ \mu,\nu\in\scr P_2.
\end{align}
\item[{\bf (a3)}] $\sup_{(x,\mu)\in\H\times \scr P}|b_t(x,\mu)|$ is locally bounded in $t$, and there exists $\phi\in\D$ such that
\beq\label{1.2}
|b_t(x,\mu)-b_t(y,\nu)|\leq \phi(|x-y|)+K(t)\W_{2}(\mu,\nu),\quad t\geq0, x,y\in \mathbb{H}, \mu,\nu \in\scr P_{2}.
\end{equation}
\end{enumerate}
\begin{rem}\label{Din}
 It is well known that $\int_0^1{\frac{\phi(s)}{s}\d s}<\infty$ is the so-called Dini condition. By \eqref{1.2}, for any $t\geq 0$ and $\mu\in\scr P_2$, $b_t(\cdot,\mu)$ is Dini continuous. Take $$
 \phi(0)=0,\ \ \ \phi(s):=\frac{K}{\log^{1+\delta}(c+s^{-1})},\ \ \ s>0
 $$ for constants $K, \delta >0$ and $c$ large enough such that $\phi^{2}$ is concave.
 Then it is routine to check $\phi\in\D$.
\end{rem}
\beg{defn}  A continuous $\F_t$-adapted process $\{X_t\}_{t\geq0}$ is called a mild solution to Equ. \eqref{E1}, if $\mathbb{P}$-a.s
\beg{equation}\label{Mil}
X_t= \e^{At} X_0+\int_{0}^{t} \e^{A(t-s)}b_s(X_s,\L_{X_s})\d s+\int_{0}^{t} \e^{A(t-s)}Q_s(X_s,\L_{X_s})\d W_s, \ \ t\geq0.
\end{equation}
Moreover, if $\E|X_t|^2<\infty$ for any $t\geq 0$, then the solution is said in $\scr P_2$.
Equ. \eqref{E1} is called strongly well-posed in $\scr P_{2}$, if for any $\F_0$-measurable random variable $X_0$ with $\L_{X_0}\in\scr P_{2}$, there exists a unique mild solution in $\scr P_2$.

(1) A couple $(\tilde{X}_t, \tilde{W}_t)_{t\geq 0}$ is called a weak solution to Equ. \eqref{E1}, if $\tilde{W}$
is a cylindrical Brownian motion with respect to a complete filtered probability space
$(\tilde{\Omega}, \{\tilde{\F}_t\}_{t\geq 0}, \tilde{\P})$, and \eqref{Mil} holds for $(\tilde{X}_t, \tilde{W}_t)_{t\geq 0}$ in place of $(X_t, W_t)_{t\geq 0}$. Moreover, if $\L_{\tilde{X}_t}|_{\tilde{\P}}\in\scr P_2$, the weak solution is called in $\scr P_2$.

(2) Equ.\eqref{E1} is said to have weak uniqueness in $\scr P_2$, if any two weak solutions in $\scr P_2$ of \eqref{E1} from common initial distribution are equal in law. Furthermore, we call weak well-posedness in $\scr P_{2}$ for Equ.\eqref{E1} holds, if it has a weak solution from any initial distribution and has weak uniqueness in $\scr P_2$.
\end{defn}

Some notations are listed which are necessary to state subsequent results and their proofs.
\begin{itemize}
\item
Let $L^2(\Omega\rightarrow\H;\F_0)$ be the class of all random variables $\xi$ which are $\F_0$-measurable and have finite second moment. Denote by $C([0,T];\H)$ and $C([0,T];\scr P_2)$ the spaces consisted of all continuous functions from $[0, T]$ to $\H$ and $\scr P_2$ respectively. Let $\B_b(\H)$ be the class of all bounded measurable functions on $\H$ and
$L^p([0,T];\H)$ be the space of the $\H$-valued functions defined on $[0,T]$ with finite $p$-th moment.
\item
For two Banach spaces $E_1$ and $E_2$ and $i=1, 2$, $C^i(E_1;E_2) (C^i_b(E_1;E_2))$ denotes the collection of all functions from $E_1$ to $E_2$ with continuous ( and bounded) Fr\'{e}chet's derivatives up to order $i$.
\item
For a real-valued or $\mathbb{H}$-valued function $f$ defined on $[0,T]\times\mathbb{H}$, let
\beg{equation*}
\|f\|_{T,\infty}=\sup_{t\in[0,T],x\in\mathbb{H}}|f(t,x)|.
\end{equation*}
Similarly, if $f$ is an operator-valued map defined on $[0,T]\times\mathbb{H}$, let
\beg{equation*}
\|f\|_{T,\infty}=\sup_{t\in[0,T],x\in\mathbb{H}}\|f(t,x)\|.
\end{equation*}
\item The letter $C$ with or without indices will denote an unimportant constant, whose values may change from one appearance to another.
\end{itemize}
This manuscript is organized as follows.
In Section 2, we state the main results, including existence and uniqueness of solutions, dimension-free Harnack inequality and shift Harnack inequality. Section 3 devotes to
proving the existence and uniqueness of solutions through the compact method and Zvonkin's transform. Using the coupling by change of measure, the proofs of Harnack inequality and shift Harnack inequality will be given in Section 4.

\section{Main results}
The first result is concerning to the weak existence under a more general frame, where the coefficients are only assumed to be bounded and continuous. From now on, let
$T$ stand for any fixed time.
\begin{thm}\label{ws} Assume {\bf(a1)}. If $\sup_{(x,\mu)\in\H\times \scr P}(|b_t(x,\mu)|+\|Q_t(x,\mu)\|)$ is locally bounded with respect to $t$ and $b_t, Q_t$ are continuous in $\H\times\scr P$ for each $t\geq0$. Then for any fixed $T>0$, and $\mu_0\in\scr P$, Equ. \eqref{E1} has a weak solution up to time $T$ with initial distribution $\mu_0$.
\end{thm}
Under {\bf (a1)}-{\bf (a3)}, the existence and uniqueness of solutions to Equ.\eqref{E1},
as well as the continuity of the solutions with respect to the initial value can be derived.
\beg{thm}\label{T2.1} Assume {\bf (a1)}-{\bf (a3)}. Then the following assertions hold.
\begin{enumerate}
\item[(1)] \eqref{E1} has weak well-posedness in $\scr P_2$. Let $P_t^\ast \mu_0$ be the unique distribution of the weak solution at time $t\geq 0$ with initial distribution $\mu_0$. There exists a constant $C(T)>0$ such that
    \begin{align}\label{Pta}
    \int_{0}^T\W_2(P_t^\ast \mu_0,P_t^\ast \nu_0)^2\d t\leq C(T)\W_2( \mu_0,\nu_0)^2,\ \ \mu_0,\nu_0\in\scr P_2.
    \end{align}
\item[(2)]   The strong well-posedness in $\scr P_2$ holds for  \eqref{E1}.
    Moreover, there exists an increasing function $C:[0,\infty)\to[0,\infty)$ such that for any two solutions $X_t$ and $Y_t$ to \eqref{E1}, it holds
   \begin{align}\label{X-Y}
   \int_0^T\E|X_s-Y_s|^2\d s\leq C(T)\E|X_0-Y_0|^2, \ \ T\geq 0.
   \end{align}
\end{enumerate}
\end{thm}
For any $\mu\in\scr P_2$ and any $f\in\B_b(\H)$, define
$$
P_t f (\mu)=(P_t^\ast\mu)(f):=\int_{\H}f\d P_t^\ast\mu,\ \ t\geq0.
$$

For a measurable space $(E, \mathcal{E})$, let $\scr P(E)$ denote
 the family of all probability measures on $(E,\mathcal{E})$.
 For $\mu,\nu\in\scr P(E)$, the relative entropy $\Ent(\nu|\mu)$
 is defined by
 $$\Ent(\nu|\mu):= \beg{cases} \int (\log \ff{\d\nu}{\d\mu})\,\d\nu, \ &\text{if}\ \nu\ \text{ is\ absolutely\ continuous\ with\ respect\ to}\ \mu,\\
 \infty,\ &\text{otherwise;}\end{cases}$$
and the total variational distance $\|\mu-\nu\|_{\operatorname{TV}}$ is defined by
$$
\|\mu-\nu\|_{\operatorname{TV}} := \sup_{A\in\mathcal{E}}|\mu(A)-\nu(A)|.
$$
By Pinsker's inequality (see \cite{Pin}),
\beq\label{ETX}
\|\mu-\nu\|_{\operatorname{TV}}^2\le \ff 1 2 \Ent(\nu|\mu),\quad \mu,\nu\in \scr P(E).
\end{equation}
Next, we consider log-Harnack inequality and Harnack inequality with power for the nonlinear semigroup $P^\ast_{t}$.
%See also \cite{Wbook,WZ} for some results on Harnack inequality for the distribution independent %SPDEs.

\beg{thm}\label{THar}
Assume {\bf (a1)}-{\bf (a3)} and that $Q_t(x,\mu)$ does not depend on $\mu$. Then the following assertions hold.
\beg{enumerate}
\item[$(1)$]
There exists an increasing function $C:[0,\infty)\to (0,\infty)$ such that for any $T>0$, the log-Harnack inequality
\begin{align*}
P_T\log f(\nu_0)
&\leq \log P_Tf(\mu_0)+\frac{C(T)}{T\wedge1}\W_2(\mu_0,\nu_0)^2,\ \ \mu_0,\nu_0\in\scr P_2
\end{align*}
holds for strictly positive function $f\in\B_{b}(\H)$. Consequently, we have
\begin{align}\label{pke}
2\|P_T^\ast\mu_0-P_T^\ast\nu_0\|^2_{\mathrm{TV}}
\leq \Ent(P_T^\ast\mu_0|P_T^\ast\nu_0)
\leq \frac{C(T)}{T\wedge1}\W_2(\mu_0,\nu_0)^2.
\end{align}
\item[$(2)$]
If $Q_t(x,\mu)$ does not depend on $(x,\mu)$, the Harnack inequality with power $p>1$ holds for non-negative $f\in\B_{b}(\H)$ and any $T>0$, i.e.
\beq\label{2.6}
(P_{T}f(\mu_0))^p
\leq P_{T} f^{p}(\nu_0)
\left( \E\exp\left\{\frac{p}{2(p-1)^2}\Phi(T)\right\}\right)^{p-1},
\ \ \mu_0,\nu_0\in \scr P_2,
\end{equation}
where
\beq
\Phi(T)=K(T)
\left(4T\phi^2\left(|X_0-Y_0|\right)+ C(T)\W_2(\mu_0,\nu_0)^2+2\frac{|X_0-Y_0|^2}{T}\right),
\end{equation}
with $\L_{X_0}=\mu_0$ and $\L_{Y_0}=\nu_0$. Consequently, $P_T^\ast\mu_0$ is  equivalent to $P_T^\ast\nu_0$ and it holds
\begin{equation}\label{ap1}
P_T\left\{\left(\frac{\d P_T^\ast\mu_0}{\d P_T^\ast\nu_0}\right)^{\frac{1}{p-1}}\right\}(\mu_0)
\leq \E\exp\left\{\frac{p}{2(p-1)^2}\Phi(T)\right\}.
\end{equation}
\end{enumerate}
\end{thm}
The next assertion characterizes the shift Harnack inequality for $P_t^\ast$.
\begin{thm} \label{TsHar}
Assume {\bf (a1)}-{\bf (a3)}. If $Q_t(x,\mu)$ does not depend on $x$, then for any $T>0$, $\mu_0\in\scr P _2$, $y\in\H$ and non-negative $f\in\B_b(\H)$, we have
\beg{equation*}
\beg{split}
P_T\log f(\mu_0)\leq\log (P_T f(\e^{AT}y+\cdot))(\mu_0)+ K(T)\left(T\phi^2(|y|)+\frac{|y|^2}{T}\right),\ \ f\geq 1,
\end{split}
\end{equation*}
and
\beg{equation*}
\beg{split}
(P_Tf(\mu_0))^p\le
&P_T(f^p(\e^{AT}y+\cdot))(\mu_0)
\exp\bigg[\ff{p}{(p-1)}K(T) \left(T\phi^2(|y|)+\frac{|y|^2}{T}\right)\bigg].
\end{split}
\end{equation*}
%where
%\begin{align*}
%\beta(T,y)&= CT\phi^2(|y|)+\frac{|y|^2}{T}.
%\end{align*}
\end{thm}

As an immediate result of Theorem \ref{TsHar} from \cite[Theorem 1.4.4]{Wbook}, we have
\begin{cor}
\label{densitys} Under the conditions of Theorem \ref{TsHar},
for each $y\in\H$ and $\mu_0\in\scr P_2$, $P_T^\ast\mu_0$ is equivalent to $(P_T^\ast\mu_0)(\cdot-\e^{AT}y)$. Moreover, for any $p>1$, it holds
$$
P_T\left\{\left(\frac{\d P_T^\ast\mu_0}{\d [(P_T^\ast\mu_0)(\cdot-\e^{AT}y)]}\right)
^{\frac{1}{p}}\right\}(\mu_0)
\leq\exp\bigg[\ff{1}{(p-1)} K(T)\left(T\phi^2(|y|)+\frac{|y|^2}{T}\right)\bigg].
$$
\end{cor}

\section{Existence and Uniqueness}
In this section, we investigate the existence and uniqueness of solutions to Equ.\eqref{E1}. Firstly,  we will use the compactness method in the proof of \cite[Theorem 8.1]{DZ} to complete the proof of Theorem \ref{ws}. Next, the strong well-posedness will be shown by the fixed point theorem combined with Zvonkin's transform which depends on distribution under {\bf (a1)-(a3)}. Finally, the weak  uniqueness can be derived by the strong well-posedness.
\subsection{Proof of Theorem \ref{ws}}
For the sake of reader's convenience, let us recall a result on the compact operators introduced in \cite[Proposition 8.4]{DZ}.
\begin{lem}\label{CO}
Let $\{S(t)\}_{t>0}$ be a family of compact operators on $\H$. Then for any $p, \alpha$ satisfying $0<\frac{1}{p}<\alpha\leq1$, the operator $G_\alpha$ defined by
\begin{align}
G_\alpha f(t)=\int_0^t(t-s)^{\alpha-1}S(t-s)f(s)\d s,\ \ t\in[0,T],
\end{align}
is compact from $L^p([0,T],\H)$ into $C([0,T],\H)$.
\end{lem}
Now, we are in the position to prove Theorem \ref{ws}.
\begin{proof}[Proof of Theorem \ref{ws}]
The proof is divided into three steps.

Step 1. For each $n\geq 1$, let $\eta_n(s)=\lfloor \frac{s}{T/n}\rfloor\frac{T}{n},$
where $\lfloor\cdot\rfloor$ stands for the integer part. Let $X_0$ be an $\F_0$-measurable random variable with $\L_{X_0}=\mu_0$.
For $t\in[0,T],$ define
\begin{align}\label{fia'}
X^n_t=\e^{At}X_0+ \int_0^t\e^{A(t-s)}b_s(X^n_{\eta_n(s)}, \L_{X^n_{\eta_n(s)}})\d s+\int_0^t\e^{A(t-s)}Q_s(X^n_{\eta_n(s)},\L_{X^n_{\eta_n(s)}})\d W_s.
\end{align}
Due to {\bf (a1)}, we have
\begin{align}\label{RHS}
\int_0^tr^{-\varepsilon}\|\e^{Ar}\|_{{\rm{HS}}}^2\d r
=\sum_{i=1}^{+\infty}\int_0^tr^{-\varepsilon}\e^{-2\lambda_i r}\d r
\leq 2^{\varepsilon-1}\Gamma(1-\varepsilon)\sum_{i=1}^{+\infty}\lambda_i^{\varepsilon-1}<\infty,
\end{align}
where $\Gamma$ stands for Gamma-function. This, together with
the condition that $b$ and $Q$ are bounded on $[0, T]$, implies that $X^n_t$ in (\ref{fia'}) is well-defined. Moreover, $X^n$ has a continuous version (see \cite[Theorem 5.9]{DZ}).

Step 2. In this step, we aim to prove $\{\L_{X^n}\}_{n\geq 1}$ is tight. Let $G_\alpha$ be as in \eqref{CO} with $\e^{At}$ in place of $S(t)$.
By (\ref{CO}), (\ref{RHS}) and stochastic Fubini theorem, we have
$$
\int_0^t\e^{A(t-s)}Q_s(X^n_{\eta_n(s)},\L_{X^n_{\eta_n(s)}})\d W_s
=\frac{\sin\frac{\varepsilon\pi}{2}}{\pi}G_{\frac{\varepsilon}{2}}Y_n(t), \ \ t\in[0,T],
$$
where
$$
Y_n(t)=\int_0^t(t-s)^{-\frac{\varepsilon}{2}}\e^{A(t-s)}Q_s(X^n_{\eta_n(s)},\L_{X^n_{\eta_n(s)}})\d W_s.
$$
Define $\tilde{G}:\mathbb{H}\to C([0,T];\H)$ as
$$[\tilde{G}(x)](t)=\e^{At}x,\ \ x\in\H,t\in[0,T].$$
It is not difficult to see that $\tilde{G}$ is a compact operator.
Then $X^n_t$ can be reformulated as
\begin{align}
X^n_t
=[\tilde{G}(X_0)](t)+G_1\left(b_\cdot(X^n_{\eta_n(\cdot)},\L_{X^n_{\eta_n(\cdot)}})\right)(t)
+\frac{\sin\frac{\varepsilon\pi}{2}}{\pi}G_{\frac{\varepsilon}{2}}Y_n(t),\ \ t\in[0,T].
\end{align}
Note that for $p>\frac{2}{\varepsilon}$ and each $n\geq1$, it is clear that
\begin{align*}
&\E\int_0^T|Y_n(t)|^p\d t
\leq C_p\int_0^T\E\left(\int_0^t(t-s)^{-\varepsilon}
\|\e^{A(t-s)}Q_s(X^n_{\eta_n(s)},\L_{X^n_{\eta_n(s)}})\|_{{\rm{HS}}}^2
\d s\right)^{\frac{p}{2}}\d t\cr
&\leq C_pT\sup\limits_{t\in[0,T]}\sup_{(x,\mu)\in\H\times \scr P}\|Q_t(x,\mu)\|^p
\left(\int_0^Tr^{-\varepsilon}\|\e^{Ar}\|_{{\rm{HS}}}^2\d r\right)^{\frac{p}{2}}
=:c_p<\infty, \ \ \forall n\geq1,
\end{align*}
where $C_p$ is a constant only depending on $p,T$ and its value can change from line to line.
%Hence, in view of $\L_{X_0}\in\scr P_p$,
Hence, we obtain
$$\P(|X_0|>r)\rightarrow0,\ \ \ r\rightarrow+\infty,$$
$$
\P\left(\int_0^T|Y_n(s)|^p\d s>r^p\right)
\leq\frac{1}{r^p}\E\int_0^T|Y_n(s)|^p\d s
\leq\frac{c_p}{r^p}\rightarrow0,\ \ \ r\rightarrow+\infty,
$$
and
\begin{align*}
&\P\left(\int_0^T|b_s(X^n_{\eta_n(s)},\L_{X^n_{\eta_n(s)}})|^p\d s>r^p\right)\leq\frac{C_p\sup\limits_{t\in[0,T]}\sup\limits_{(x,\mu)\in\H\times \scr P}|b_t(x,\mu)|^pT}{r^p}\rightarrow0,\ \ \ r\rightarrow+\infty.
\end{align*}
Therefore, for each $\delta>0$ small enough, there exists $r_\delta>0$ such that
$$
\P\left(|X_0|\leq r_\delta,\ \ \left(\int_0^T|Y_n(s)|^p\d s\right)^{\frac{1}{p}}\leq r_\delta, \ \
\left(\int_0^T|b_s(X^n_{\eta_n(s)},\L_{X^n_{\eta_n(s)}})|^p\d s\right)^{\frac{1}{p}}\leq r_\delta\right)\geq1-\delta.
$$
This leads to $\L_{X^n}(K_\delta)\geq 1-\delta, n\geq 1$, where
$$
K_\delta:=\left\{\tilde{G}x+G_1f+\frac{\sin\frac{\varepsilon\pi}{2}}{\pi}G_{\frac{\varepsilon}{2}} g: |x|\leq r_\delta, \left(\int_0^T|f(s)|^p\d s\right)^{\frac{1}{p}}\leq r_\delta, \left(\int_0^T|g(s)|^p\d s\right)^{\frac{1}{p}}\leq r_\delta \right\}
$$
is compact by Lemma \ref{CO}.  So $\{\L_{X^n}\}_{n\geq 1}$ is tight.

Step 3. Due to the tightness of $\{\L_{X^n}\}_{n\geq 1}$, there exists a weakly convergent subsequence still denoted by $\{\L_{X^n}\}_{n\geq 1}$. By the Skorohod representation theorem \cite[Theorem 2.4]{DZ}, there exists a probability space $(\tilde{\Omega},\tilde{\F}, \tilde{\P})$ and $C([0,T];\H)$-valued stochastic processes $\tilde{X}^{n}$, $\tilde{X}$ such that $\L_{X^{n}}|_\P=\L_{\tilde{X}^{n} }|_{\tilde{\P}}$, and $\tilde{\P}$-a.s. $\tilde{X}^{n}$ converges to $\tilde{X}$ as $n\to\infty$, which implies that for any $t\in[0,T]$, $\L_{\tilde{X}^{n}_t}|_{\tilde{\P}}$ weakly converges to $\L_{\tilde{X}_t}|_{\tilde{\P}}$. On the other hand, it follows from \eqref{fia'} that
\begin{align}\label{fia''}
(-A)^{-1}X^n_t=&\e ^{At}(-A)^{-1}X_0 +\int_0^t\e^{A(t-s)}(-A)^{-1}b_s(X^n_{\eta_n(s)}, \L_{X^n_{\eta_n(s)}})\d s\\ \nonumber
&+\int_0^t\e^{A(t-s)}(-A)^{-1}Q_s(X^n_{\eta_n(s)},\L_{X^n_{\eta_n(s)}})\d W_s, \ \ t\in[0,T].
\end{align}
In view of \cite[Lemma 3.5]{BG}, \eqref{fia''} implies
\begin{align}\label{fi'}
(-A)^{-1}X^n_t=&(-A)^{-1}X_0+ \int_0^t(-X^n_s)\d s+\int_0^t(-A)^{-1}b_s(X^n_{\eta_n(s)}, \L_{X^n_{\eta_n(s)}})\d s\\ \nonumber
&+\int_0^t(-A)^{-1}Q_s(X^n_{\eta_n(s)},\L_{X^n_{\eta_n(s)}})\d W_s.
\end{align}
Let
\begin{align*}
N^n_t:=(-A)^{-1}X^n_t-(-A)^{-1}X_0+ \int_0^tX^n_s\d s-\int_0^t(-A)^{-1}b_s(X^n_{\eta_n(s)}, \L_{X^n_{\eta_n(s)}})\d s, \ \ t\in[0,T],
\end{align*}
and $\tilde{N}^n$ be defined in the same way with $X^n$ replaced by $\tilde{X}^n$.
It is clear that $\{N^n_t\}_{t\in[0,T]}$ is a martingale with respect to the filtration $\F^n_t=\sigma\{X^n_s,s\leq t\}$. Thanks to $\L_{X^n}|_{\P}=\L_{\tilde{X}^n}|_{\tilde{\P}}$ and the boundedness of $Q$ and $b$, it is not difficult to prove that $\{\tilde{N}^n_t\}_{t\in[0,T]}$ is a martingale with respect to the filtration $\tilde{\F}^n_t=\sigma\{\tilde{X}^n_s,s\leq t\}$ and  the  quadratic variation process is
$$\<\tilde{N}^n\>_t=\int_0^t\left((-A)^{-1}Q_s(\tilde{X}^n_{\eta_n(s)},\L_{\tilde{X}^n_{\eta_n(s)}})\right) \left((-A)^{-1}Q_s(\tilde{X}^n_{\eta_n(s)},\L_{\tilde{X}^n_{\eta_n(s)}})\right)^\ast\d s,\ \ t\in[0,T],$$
where $\ast$ stands for the adjoint operator.
Noting that
$$|\tilde{X}^n_{\eta_n(s)}-\tilde{X}_{s}|\leq |\tilde{X}^n_{\eta_n(s)}-\tilde{X}_{\eta_n(s)}|+|\tilde{X}_{\eta_n(s)}-\tilde{X}_{s}|\leq \sup_{s\in[0,T]}|\tilde{X}^n_{s}-\tilde{X}_{s}|+|\tilde{X}_{\eta_n(s)}-\tilde{X}_{s}|,$$
we conclude that $\tilde{\P}$-a.s. $\tilde{X}^n_{\eta_n(s)}$ converges to $\tilde{X}_{s}$ as $n$ goes to infinity. This combined with the continuity of $b_t,Q_t$ implies that the process
$$\tilde{N}_t:=(-A)^{-1}\tilde{X}_t-(-A)^{-1}\tilde{X}_0+ \int_0^t\tilde{X}_s\d s-\int_0^t(-A)^{-1}b_s(\tilde{X}_{s}, \L_{\tilde{X}_{s}})\d s,\ \ t\in[0,T]$$
is a martingale with respect to the filtration $\tilde{\F}_t=\sigma\{\tilde{X}_s,s\leq t\}$ and
the quadratic variation process is
$$\<\tilde{N}\>_t=\int_0^t\left((-A)^{-1}Q_s(\tilde{X}_{s},\L_{\tilde{X}_{s}})\right) \left((-A)^{-1}Q_s(\tilde{X}_{s},\L_{\tilde{X}_{s}})\right)^\ast\d s,\ \ t\in[0,T].$$
By the martingale representation theorem \cite[Theorem 8.2]{DZ}, there exists a complete filtered probability space $(\tilde{\tilde{\Omega}},\tilde{\tilde{\F}},\{\tilde{\tilde{\F}}_t\},\tilde{\tilde{\P}})$, a cylindrical Brownian motion $\tilde{\tilde{W}}$
such that $\L_{\tilde{X}_{s}}|_{\tilde{\P}}=\L_{\tilde{X}_{s}}|_{\tilde{\tilde{\P}}}$ and
\begin{align}\label{f}
(-A)^{-1}\tilde{X}_t=&(-A)^{-1}\tilde{X}_0+ \int_0^t(-\tilde{X}_s)\d s+\int_0^t(-A)^{-1}b_s(\tilde{X}_{s}, \L_{\tilde{X}_{s}}|_{\tilde{\tilde{\P}}})\d s\\ \nonumber
&+\int_0^t(-A)^{-1}Q_s(\tilde{X}_{s}, \L_{\tilde{X}_{s}}|_{\tilde{\tilde{\P}}})\d \tilde{\tilde{W}}_s, \ \ t\in[0,T].
\end{align}
Again by \cite[Lemma 3.5]{BG}, \eqref{f} yields
\begin{align*}
(-A)^{-1}\tilde{X}_t=&\e^{At}(-A)^{-1}\tilde{X}_0+\int_0^t(-A)^{-1}\e^{A(t-s)}b_s(\tilde{X}_{s}, \L_{\tilde{X}_{s}}|_{\tilde{\tilde{\P}}})\d s\\ \nonumber
&+\int_0^t(-A)^{-1}\e^{A(t-s)}Q_s(\tilde{X}_{s}, \L_{\tilde{X}_{s}}|_{\tilde{\tilde{\P}}})\d \tilde{\tilde{W}}_s,\ \ t\in[0,T],
\end{align*}
which derives
\begin{align}\label{fig}
\tilde{X}_t&= \e ^{At}\tilde{X}_0+\int_0^t\e^{A(t-s)}b_s(\tilde{X}_{s}, \L_{\tilde{X}_{s}}|_{\tilde{\tilde{\P}}})\d s+\int_0^t\e^{A(t-s)}Q_s(\tilde{X}_{s}, \L_{\tilde{X}_{s}}|_{\tilde{\tilde{\P}}})\d \tilde{\tilde{W}}_s,\ \ t\in[0,T].
\end{align}
Thus, $(\tilde{X}_t,\tilde{\tilde{W}}_t)_{t\in[0,T]}$ is a weak solution of \eqref{E1} with initial distribution $\mu_0$.
\end{proof}
\subsection{Strong Well-posedness under {\bf (a1)-(a3)}}
In this part, the Zvonkin transform is used to obtain the strong well-posedness.
Since Ito's formula for \eqref{E1} is unavailable, we shall use finite dimensional approximation such that the It\^o's formula can be applied. To this end, for $\lambda >0$, $\mu\in C([0,T],\scr P_2)$ and $n\geq 1$, let $\H_n=\mathrm{span}\{e_1,\cdots,e_n\}$ and define
$$b^{\mu,n}_t=\pi_n b_t(\cdot,\mu_t)\circ \pi_n,\ \ Q^{\mu,n}_t=\pi_n Q_t(\cdot,\mu_t)\circ \pi_n,\ \ A_n=A\circ \pi_n.$$
Let $Z^n_{s,t}(z)$ solve
\beq\label{E-A-n}
\d Z^n_t=A_n Z^n_t\d t+Q^{\mu,n}_t(Z^n_t)\d W_t
\end{equation}
with $Z^n_{s,s}(z)=z\in\H_n$ and
$P_{s,t}^{\mu,n}$ be the associated semigroup. That is,
\begin{equation*}
P_{s,t}^{\mu,n}f(x)=\mathbb{E}f(Z^n_{s,t}(x)), \ \ x\in\H_n, f\in\B_b(\mathbb{H}_n), t\geq s\geq 0.
\end{equation*}
Consider
\beq\label{un}
u^n_s=\int_{s}^{T} \e^{-\lambda(t-s)}P_{s,t}^{\mu,n}(\nabla_{b^{\mu,n}_t}u^n_t+b^{\mu,n}_t)\d t, \ \ s\in[0,T].
\end{equation}
Due to \cite[Lemma 2.3, Proposition 2.5]{W}, we have
\beg{lem}\label{L-PDE}
Assume {\bf (a1)-(a3)}. Let $T>0$ be fixed. Then there exists a constant $\lambda_0>0$ independent of $n$ such that for any $\lambda\geq\lambda_0$, \eqref{un} has a unique solution $u^{\lambda,\mu,n}$ which belongs to $C^1([0,T]; C_{b}^{2}(\mathbb{H}_n; \mathbb{H}_n))$ with
\beq\label{g1'}
\| u^{\lambda,\mu,n}\|_{T,\infty}+\|\nabla  u^{\lambda,\mu,n}\|_{T,\infty}+\left\|\nabla^2 u^{\lambda,\mu,n}\right\|_{T,\infty}\leq \frac{1}{5},\ \ n\geq 1.
\end{equation}
\end{lem}
Let $\Theta^{\lambda,\mu,n}(x)=x+u^{\lambda,\mu,n}(x),x\in\H_n$.
Then we have the regularization of the finite dimensional approximation as follows.
\beg{lem}\label{L3.2} Assume {\bf (a1)-(a3)}. For any $T>0$, there exists a constant $\lambda(T)\geq \lambda_0$ such that for any $\zeta\in C([0,T];\scr P_2)$ and adapted continuous process $(X_t)_{t\in[0,T]}$ on $\mathbb{H}$ with $\mathbb{P}$-a.s.
\beq\label{3.3}\beg{split}
&X_t= \e^{At} X_0+\int_{0}^{t} \e^{A(t-s)}b_s(X_s,\zeta_s)\d s+\int_{0}^{t} \e^{A(t-s)}Q_s(X_s,\zeta_s)\d W_s,\ \ t\in[0,T],
\end{split}\end{equation}
and any $\lambda\geq\lambda(T), n\geq 1$, $X_t^n:=\pi_n X_t$ satisfies
\beq\label{3.4} \beg{split}
&\Theta_t^{\lambda,\mu,n}(X^n_t)= \e^{At}\Theta_0^{\lambda,\mu,n}(X^n_0)
+\int_{0}^{t} \e^{A(t-s)}\nabla \Theta_s^{\lambda,\mu,n}(X^n_s)\pi_n Q_s(X_s,\zeta_s)\d W_s\\
&+\int_{0}^{t}(\lambda-A)\e^{A(t-s)}u^{\lambda,\mu,n}_s(X^n_s)\d s\\
&+\int_0^t\e^{A(t-s)}\nabla \Theta_s^{\lambda,\mu,n}(X^n_s)\pi_n[b_s(X_s,\zeta_s)-b_s(X^n_s,\mu_s)]\d s\\
&+\frac{1}{2}\int_0^t\e^{A(t-s)}\mathrm{tr}\{[(Q_sQ^\ast_s)(X_s,\zeta_s)-(Q_sQ^\ast_s)(X^n_s,\mu_s)]\nabla ^2 u_s^{\lambda,\mu,n}(X^n_s)\}\d s,\ \ t\in[0,T],
\end{split}\end{equation}
where
\begin{align*}
&\e^{A(t-s)}\mathrm{tr}\{[(Q_sQ^\ast_s)(X_s,\zeta_s)-\<Q_sQ^\ast_s)(X^n_s,\mu_s)]\nabla ^2 u_s^{\lambda,\mu,n}(X^n_s)\}\\
&:=\sum_{i=1}^n\left( \e^{-\lambda_i(t-s)}\mathrm{tr}\{[(Q_sQ^\ast_s)(X_s,\zeta_s)-(Q_sQ^\ast_s)(X^n_s,\mu_s)]\nabla ^2 \<u_s^{\lambda,\mu,n}(X^n_s),e_i\>\}\right)e_i.
\end{align*}
\end{lem}
\begin{proof} The proof mainly follows the idea of \cite[Proposition 2.5]{W}. However, due to the distribution dependence of $b$ and $Q$, much more work needs to be done.

For simplicity, let $b_t^\mu=b_t(\cdot,\mu_t)$ and $Q_t^\mu=Q_t(\cdot,\mu_t)$. For any second-order differential function $F$ on $\H_n$, let $L_t^{\mu,n}$ be defined as
\begin{align}\label{gen}L_t^{\mu,n}F(z)&=\<Az,\nabla F(z)\>+\frac{1}{2}\sum_{i,j=1}^n\<(Q^{\mu}_t(Q^{\mu}_t)^\ast)(z) e_i,e_j\>\nabla_{e_i}\nabla_{e_j}F(z)
,\ \ z\in\H_n.
\end{align}
%For any $z\in\H$, define $\tilde{u}^{\lambda,\mu,n}(z)=u^{\lambda,\mu,n}(\pi_nz)$. For simplicity, we still denote $\tilde{u}^{\lambda,\mu,n}$ by $u^{\lambda,\mu,n}$.
This together with \cite[(2.6)]{W}, dominated convergence theorem and $u^{\lambda,\mu,n}=\pi_n u^{\lambda,\mu,n}\circ\pi_n$ implies
\begin{align}\label{put}
\partial_s u^{\lambda,\mu,n}_s(z)=[(\lambda-L_s^{\mu,n}) u^{\lambda,\mu,n}_s](z)-[\nabla_{b^{\mu,n}_s} u^{\lambda,\mu,n}_s+b^{\mu,n}_s](z),\ \ z\in\H_n.
\end{align}
Since $X_s^n=\pi_nX_s$ solves the following equation
\begin{align}\label{Xn}\d X^n_s= AX^n_s\d s+\pi_nb_s(X_s,\zeta_s)\d s+\pi_nQ_s(X_s,\zeta_s)\d W_s,\ \ s\in[0,T],
\end{align}
It\^{o}'s formula, \eqref{put}, $ u^{\lambda,\mu,n}_s= \pi_n u^{\lambda,\mu,n}_s\circ \pi_n$ and $[\nabla_{b^{\mu,n}_s} u^{\lambda,\mu,n}_s]\circ\pi_n=[\nabla_{b^{\mu}_s} u^{\lambda,\mu,n}_s]\circ\pi_n$ lead to
\begin{align*}
\d  u^{\lambda,\mu,n}_s(X^n_s)
=&\<\nabla u^{\lambda,\mu,n}_s(X^n_s), Q_s(X_s,\zeta_s)\d W_s\>+\partial_s u^{\lambda,\mu,n}_s(X^n_s)\d s\\
&+\<\nabla u^{\lambda,\mu,n}_s(X^n_s),b_s(X_s,\zeta_s)\>\d s+\<AX^n_s,\nabla u^{\lambda,\mu,n}_s(X^n_s)\>\d s\\
&+\frac{1}{2}\sum_{i,j=1}^n\<(Q_sQ^\ast_s)(X_s,\zeta_s) e_i,e_j\>\nabla_{e_i}\nabla_{e_j}u^{\lambda,\mu,n}_s(X^n_s)\d s\\
=&\<\nabla u^{\lambda,\mu,n}_s(X^n_s), Q_s(X_s,\zeta_s)\d W_s\>+\lambda u^{\lambda,\mu,n}_s(X_s^n)\d s-[L_s^{\mu,n} u^{\lambda,\mu,n}_s](X_s^n)\d s\\
&-[\nabla_{b^{\mu,n}_s} u^{\lambda,\mu,n}_s+b^{\mu,n}_s](X_s^n)\d s+\<\nabla u^{\lambda,\mu,n}_s(X^n_s),b_s(X_s,\zeta_s)\>\d s\\
&+\<AX^n_s,\nabla u^{\lambda,\mu,n}_s(X^n_s)\>\d s+\frac{1}{2}\sum_{i,j=1}^n\<(Q_sQ^\ast_s)(X_s,\zeta_s) e_i,e_j\>\nabla_{e_i}\nabla_{e_j}u^{\lambda,\mu,n}_s(X^n_s)\d s\\
=&\<\nabla u^{\lambda,\mu,n}_s(X^n_s), Q_s(X_s,\zeta_s)\d W_s\>+\lambda u^{\lambda,\mu,n}_s(X^n_s)\d s\\
&+\<\nabla u^{\lambda,\mu,n}_s(X^n_s),b_s(X_s,\zeta_s)-b^\mu_s(X_s^n)\>\d s-\pi_nb^{\mu}_s(X_s^n)\d s\\
&+\frac{1}{2}\sum_{i,j=1}^n\<[(Q_sQ^\ast_s)(X_s,\zeta_s)- Q^{\mu}_s(Q^{\mu}_s)^\ast(X_s^n)] e_i,e_j\>\nabla_{e_i}\nabla_{e_j}u^{\lambda,\mu,n}_s(X_s^n)\d s.
\end{align*}
This together with \eqref{Xn} and $ u^{\lambda,\mu,n}_s= \pi_n u^{\lambda,\mu,n}_s\circ \pi_n$ yields
\begin{align*}
&\d  [u^{\lambda,\mu,n}_s(X^n_s)+X^n_s]\\
=&A[X^n_s+u^{\lambda,\mu,n}_s(X^n_s)]\d s+(\lambda-A) u^{\lambda,\mu,n}_s(X^n_s)\d s\\
&+\<\nabla u^{\lambda,\mu,n}_s(X^n_s), Q_s(X_s,\zeta_s)\d W_s\>+\pi_nQ_s(X_s,\zeta_s)\d W_s\\
&+\<\nabla u^{\lambda,\mu,n}_s(X^n_s),b_s(X_s,\zeta_s)-b^\mu_s(X_s^n)\>\d s+[\pi_nb_s(X_s,\L_{X_s})-\pi_nb^{\mu}_s(X_s^n)]\d s\\
&+\frac{1}{2}\sum_{i,j=1}^n\<[(Q_sQ^\ast_s)(X_s,\zeta_s)- (Q^{\mu}_s(Q^{\mu}_s)^\ast)(X_s^n)] e_i,e_j\>\nabla_{e_i}\nabla_{e_j}u^{\lambda,\mu,n}_s(X_s^n)\d s.
\end{align*}
Thus, we get
\begin{align}\label{fi} \nonumber&u^{\lambda,\mu,n}_t(X^n_t)+X^n_t\\ \nonumber
=&\e^{At}[u_0^{\lambda,\mu,n}(X^n_0)+X^n_0]+\int_0^t\e^{A(t-s)}(\lambda-A) u^{\lambda,\mu,n}_s(X^n_s)\d s\\ \nonumber
&+\int_0^t\e^{A(t-s)}\<\nabla u^{\lambda,\mu,n}_s(X^n_s)+I, \pi_nQ_s(X_s,\zeta_s)\d W_s\>\\
&+\int_0^t\e^{A(t-s)}\<\nabla u^{\lambda,\mu,n}_s(X^n_s)+I,\pi_nb_s(X_s,\zeta_s)-\pi_nb^\mu_s(X_s^n)\>\d s\\ \nonumber
&+\frac{1}{2}\int_0^t\e^{A(t-s)}\sum_{i,j=1}^n\<(Q_sQ^\ast_s)(X_s,\zeta_s)- (Q^{\mu}_s(Q^{\mu}_s)^\ast)(X_s^n) e_i,e_j\>\nabla_{e_i}\nabla_{e_j}u^{\lambda,\mu,n}_s(X_s^n)\d s.\nonumber
\end{align}
The proof is finished.
\end{proof}
\begin{rem}\label{mor}
The conditions in Lemma \ref{L3.2} are stronger than those required in \cite[Proposition 2.5]{W}, where {\bf(a3'')} cannot ensure the existence of $\nabla^2 u^{\lambda,\mu,n}$ in \eqref{3.4}. Moreover, different from the proof of \cite[Proposition 2.5 ]{W}, we do not take limit in \eqref{3.4} with respect to $n$ in order to avoid calculating $\lim_{n \to\infty}\nabla^2u^{\lambda,\mu,n}_s(X_s^n)$. However, it is enough to prove the strong well-posedness by \eqref{3.4}, see the proof of Theorem \ref{T2.1}(2) below for  more details.
\end{rem}
%\begin{thm}\label{paw} Assume {\bf(a1)}-{\bf (a3)}. Then the strong well-posedness for \eqref{E1} in $\scr P_2$ holds. Moreover, for any two solutions $X_t$ and $Y_t$ with initial values $X_0, Y_0\in L^2(\Omega\rightarrow\H;\F_0)$ of Euq.\eqref{E1}, and for any $T>0$, there exists a constant $C(T)>0$ such that
%\begin{equation}
%\label{di}\int_0^T\E|X_t-Y_t|^2\d t\leq C(T)\E|X_0-Y_0|^2.
%\end{equation}
%Consequently, $X_0=Y_0$ implies $\P$-a.s. $X_t=Y_t, t\in[0,T]$.
%\end{thm}
\beg{proof}[Proof of Theorem \ref{T2.1}(2)] According to \cite[Theorem 1.1]{W}, for any $\mu\in C([0,T];\scr P_2)$ and $X_0\in L^2(\Omega\rightarrow\H;\F_0)$, the following equation
\begin{align}
\label{DD}\d X_t=\{A X_t+b_t(X_t, \mu_t)\}\d t+Q_t(X_t,\mu_t)\d W_t
\end{align}
has a unique mild solution $X_t$. Let $\nu\in C([0,T];\scr P_2)$ and $Y_0\in L^2(\Omega\rightarrow\H;\F_0)$ and $Y_t$ solve \eqref{DD} with $(\mu,X_0)$ replaced by $(\nu,Y_0)$. Moreover, let $\Phi_t^{\L_{X_0}}(\mu)$ and $\Phi_t^{\L_{Y_0}}(\nu)$ be the distribution of $X_t$ and $Y_t$ respectively. Set $X_t^n=\pi_n X_t$ and $Y_t^n=\pi_n Y_t$.

Let $\lambda$ be large enough such that the assertions in Lemma \ref{L3.2} and Lemma \ref{L-PDE} hold. By \eqref{3.4}, we have $\mathbb{P}$-a.s.
\beq\label{3.6}\beg{split}
&\Theta_t^{\lambda,\mu,n}(X^n_t)-\Theta_t^{\lambda,\mu,n}(Y^n_t)\\
=&\e^{At}\left(\Theta_0^{\lambda,\mu,n}(X^n_0)-\Theta_0^{\lambda,\mu,n}(Y^n_0)\right)\\
&+\int_{0}^{t} \e^{A(t-s)}[\nabla \Theta_s^{\lambda,\mu,n}(X^n_s)\pi_n Q_s(X_s,\mu_s)-\nabla \Theta_s^{\lambda,\mu,n}(Y^n_s)\pi_n Q_s(Y_s,\nu_s)]\d W_s\\
&+\int_{0}^{t}(\lambda-A)\e^{A(t-s)}[u^{\lambda,\mu,n}_s(X^n_s)-u^{\lambda,\mu,n}_s(Y^n_s)]\d s\\
&+\int_0^t\e^{A(t-s)}\nabla \Theta_s^{\lambda,\mu,n}(X^n_s)\pi_n[b_s(X_s,\mu_s)-b_s(X^n_s,\mu_s)]\d s\\
&+\frac{1}{2}\int_0^t\e^{A(t-s)}\mathrm{tr}\{[(Q_sQ^\ast_s)(X_s,\mu_s)-(Q_sQ^\ast_s)(X^n_s,\mu_s)]\nabla ^2 u_s^{\lambda,\mu,n}(X^n_s)\}\d s\\
&-\int_0^t\e^{A(t-s)}\nabla \Theta_s^{\lambda,\mu,n}(Y^n_s)\pi_n[b_s(Y_s,\nu_s)-b_s(Y^n_s,\mu_s)]\d s\\
&-\frac{1}{2}\int_0^t\e^{A(t-s)}\mathrm{tr}\{[(Q_sQ^\ast_s)(Y_s,\nu_s)-(Q_sQ^\ast_s)(Y^n_s,\mu_s)]\nabla ^2 u_s^{\lambda,\mu,n}(Y^n_s)\}\d s,\ \ t\in[0, T].\end{split}\end{equation}
%Firstly, let
%\beq\label{3.9}
%\eta_{l}:=\int_{0}^{l}\e^{-2\lambda t}\mathbb{E}|X_t-Y_t|^{2}\d t.
%\end{equation}
By the same argument as in \cite[(3.7)]{W} and Fatou's lemma, it is routine to obtain
\beg{equation*}
\beg{split}
&\E\liminf_{n\to\infty}\int_{0}^{l}\e^{-2\lambda t}\left|\int_{0}^{t}(\lambda-A)\e^{A(t-s)}(u^{\lambda,\mu,n}_s(X^n_s)-u^{\lambda,\mu,n}_s(Y^n_s))\d s\right|^{2}\d t\\
&\leq \liminf_{n\to\infty}\E\int_{0}^{l}\e^{-2\lambda t}\left|\int_{0}^{t}(\lambda-A)\e^{A(t-s)}(u^{\lambda,\mu,n}_s(X^n_s)-u^{\lambda,\mu,n}_s(Y^n_s))\d s\right|^{2}\d t\\
&\leq \frac{1}{4}\int_{0}^{l}\e^{-2\lambda t}\mathbb{E}|X_t-Y_t|^{2}\d t, \ \ l\in[0,T].
\end{split}\end{equation*}
Due to {\bf (a2)}, Fatou's lemma and Lemma \ref{L-PDE}, there exists some function $\varepsilon(\lambda)\downarrow 0$ as $\lambda\uparrow\infty$ such that
\beg{equation*}\begin{split}
&\mathbb{E}\liminf_{n\to\infty}\int_{0}^{l}\e^{-2\lambda t}\left|\int_{0}^{t} \e^{A(t-s)}[\nabla \Theta_s^{\lambda,\mu,n}(X^n_s)\pi_n Q_s(X_s,\mu_s)-\nabla \Theta_s^{\lambda,\mu,n}(Y^n_s)\pi_n Q_s(Y_s,\nu_s)]\d W_s\right|^{2}\d t\\
&\leq \liminf_{n\to\infty}\int_{0}^{l}\e^{-2\lambda t}\mathbb{E}\left|\int_{0}^{t} \e^{A(t-s)}[\nabla \Theta_s^{\lambda,\mu,n}(X^n_s)Q_s(X_s,\mu_s)-\nabla \Theta_s^{\lambda,\mu,n}(Y^n_s) Q_s(Y_s,\nu_s)]\d W_s\right|^{2}\d t\\
&\leq\varepsilon(\lambda)\int_{0}^{l} \e^{-2\lambda s}\mathbb{E}|X_s-Y_s|^{2}\d s+\varepsilon(\lambda)\int_{0}^{l} \e^{-2\lambda s}\W_2(\mu_s,\nu_s)^{2}\d s, \ \ l\in[0,T].
\end{split}
\end{equation*}
Furthermore, it follows from {\bf(a2)}-{\bf(a3)}, Lemma \ref{L-PDE} and dominated convergence theorem that
\begin{align*}
&\E\liminf_{n\to\infty}\int_{0}^{l}\e^{-2\lambda t}\left|\int_0^t\e^{A(t-s)}\nabla \Theta_s^{\lambda,\mu,n}(Y^n_s)\pi_n[b_s(Y_s,\nu_s)-b_s(Y^n_s,\mu_s)]\d s\right|^2\d t\\
&\leq \tilde{\varepsilon}(\lambda)\int_0^l\e^{-2\lambda s}\W_2(\mu_s,\nu_s)^{2}\d s+c\E\lim_{n\to\infty}\int_{0}^{l}\left|b_s(Y_s,\mu_s)-b_s(Y^n_s,\mu_s)\right|^2\d s\\
&=\tilde{\varepsilon}(\lambda)\int_0^l\e^{-2\lambda s}\W_2(\mu_s,\nu_s)^{2}\d s,\ \ l\in[0,T],
\end{align*}
and
\begin{align*}
&\E\liminf_{n\to\infty}\int_{0}^{l}\e^{-2\lambda t}\left|\int_0^t\e^{A(t-s)}\mathrm{tr}\{[(Q_sQ^\ast_s)(Y_s,\nu_s)-(Q_sQ^\ast_s)(Y^n_s,\mu_s)]\nabla ^2 u_s^{\lambda,\mu,n}(Y^n_s)\}\d s\right|^2\d t\\
\leq&\E\int_{0}^{l}\e^{-2\lambda t}\int_0^t\|\e^{A(t-s)}\|^2_{\rm{HS}}
\left|\mathrm{tr}[(Q_sQ^\ast_s)(Y_s,\nu_s)-(Q_sQ^\ast_s)(Y_s,\mu_s)]\right|^2\d s\d t\\
&+\E\liminf_{n\to\infty}\int_{0}^{l}\e^{-2\lambda t}\int_0^t\|\e^{A(t-s)}\|^2_{\rm{HS}}
\left|\mathrm{tr}[(Q_sQ^\ast_s)(Y_s,\mu_s)-(Q_sQ^\ast_s)(Y^n_s,\mu_s)]\right|^2\d s\d t\\
\leq& 2K(T)^2\E\int_{0}^{l}\e^{-2\lambda t}\int_0^t\|\e^{A(t-s)}\|^2_{\rm{HS}}
\left\|Q_s(Y_s,\nu_s)-Q_s(Y_s,\mu_s)\right\|^2_{\rm{HS}}\d s\d t\\
&+2K(T)^2\E\liminf_{n\to\infty}\int_{0}^{l}\e^{-2\lambda t}\E\int_0^t\|\e^{A(t-s)}\|^2_{\rm{HS}}
\left\|Q_s(Y_s,\mu_s)-Q_s(Y^n_s,\mu_s)\right\|^2_{\rm{HS}}\d s\d t\\
\leq&2K(T)^3\int_{0}^{l}\e^{-2\lambda t}\int_0^t\|\e^{A(t-s)}\|^2_{\rm{HS}}\W_2(\mu_s,\nu_s)^{2}\d s\d t\\
&+2K(T)^2\E\liminf_{n\to\infty}\int_{0}^{l}\e^{-2\lambda t}\int_0^t\|\e^{A(t-s)}\|^2_{\rm{HS}}
\left\|Q_s(Y_s,\mu_s)-Q_s(Y^n_s,\mu_s)\right\|^2_{\rm{HS}}\d s\d t\\
\leq&2K(T)^3\int_{0}^{l}\e^{-2\lambda s}\W_2(\mu_s,\nu_s)^{2}\d s\int_s^l\e^{-2\lambda (t-s)}\|\e^{A(t-s)}\|^2_{\rm{HS}}\d t\\
&+2K(T)^2\E\liminf_{n\to\infty}\int_{0}^{l}\left\|Q_s(Y_s,\mu_s)-Q_s(Y^n_s,\mu_s)\right\|^2_{\rm{HS}}\d s
\int_s^l\e^{-2\lambda (t-s)}\|\e^{A(t-s)}\|^2_{\rm{HS}}\d t\\
\leq& \tilde{\varepsilon}(\lambda)\int_{0}^{l}\e^{-2\lambda s}\W_2(\mu_s,\nu_s)^{2}\d s+c\E\lim_{n\to\infty}\int_{0}^{l}\left\|Q_s(Y_s,\mu_s)-Q_s(Y^n_s,\mu_s)\right\|^2_{\rm{HS}}\d s\\
=&\tilde{\varepsilon}(\lambda)\int_{0}^{l}\e^{-2\lambda s}\W_2(\mu_s,\nu_s)^{2}\d s, \ \ l\in[0,T]
\end{align*}
for some $\tilde{\varepsilon}(\lambda)\downarrow 0$ as $\lambda\uparrow\infty$, where we use \eqref{1.3} in the last display.
Similarly, dominated convergence theorem, {\bf(a3)}, Lemma \ref{L-PDE} and \eqref{1.3} lead to
\begin{align*}
\E\liminf_{n\to\infty}\int_{0}^{l}\e^{-2\lambda t}\left|\int_0^t\e^{A(t-s)}\nabla \Theta_s^{\lambda,\mu,n}(X^n_s)\pi_n[b_s(X_s,\mu_s)-b_s(X^n_s,\mu_s)]\d s\right|^2\d t=0,
\end{align*}
and
\begin{align*}
\E\liminf_{n\to\infty}\int_{0}^{l}\e^{-2\lambda t}\left|\int_0^t\e^{A(t-s)}\mathrm{tr}\{[(Q_sQ^\ast_s)(X_s,\mu_s)-(Q_sQ^\ast_s)(X^n_s,\mu_s)]\nabla ^2 u_s^{\lambda,\mu,n}(X^n_s)\}\d s\right|^2\d t=0.
\end{align*}
Finally, by the monotone convergence theorem and Lemma \ref{L-PDE}, we arrive at
\begin{align*}
&\E\liminf_{n\to\infty}\int_{0}^{l}\e^{-2\lambda t}\left|\Theta_t^{\lambda,\mu,n}(X^n_t)-\Theta_t^{\lambda,\mu,n}(Y^n_t)\right|^2\d t\\
&\geq\frac{16}{25}\E\lim_{n\to\infty}\int_{0}^{l}\e^{-2\lambda t}\left|X^n_t-Y^n_t\right|^2\d t=\frac{16}{25}\E\int_{0}^{l}\e^{-2\lambda t}\left|X_t-Y_t\right|^2\d t.
\end{align*}
Combining all the estimates above, for $\lambda$ large enough, we have
%and noting that $$\W_2(\mu_s,\nu_s)^2\leq \E|X_s-Y_s|^2,$$ when $\lambda$ is large enough, there exists a constant $c(T)>0$ such that
%By \cite[(2.44)]{W}, we have
%\begin{align}\label{unl}
%\lim_{n\to\infty}u^{\lambda,\mu,n}\circ\pi_n=u^{\lambda,\mu},\ \ \lim_{n\to\infty}\int_0^T\|\nabla u_s^{\lambda,\mu,n}\circ\pi_n-\nabla u^{\lambda,\mu}_s\|\d s=0.
%\end{align}

\begin{align}\label{ga}
\int_{0}^{l}\e^{-2\lambda s}\E|X_s-Y_s|^2\d s\leq \frac{1}{2}\int_{0}^{l}\e^{-2\lambda s}\W_2(\mu_s,\nu_s)^2\d s+c(T)\E|X_0-Y_0|^2, \ \ l\in[0,T].
\end{align}
This combined with $\W_2(\Phi_s^{\L_{X_0}}(\mu),\Phi_s^{\L_{Y_0}}(\nu))^2\leq\E|X_s-Y_s|^2 $ implies for $\lambda$ large enough, it holds
\begin{align}\label{gac}
\int_{0}^{T}\e^{-2\lambda s}\W_2(\Phi_s^{\L_{X_0}}(\mu),\Phi_s^{\L_{Y_0}}(\nu))^2\d s\leq \frac{1}{2}\int_{0}^{T}\e^{-2\lambda s}\W_2(\mu_s,\nu_s)^2\d s+c(T)\E|X_0-Y_0|^2.
\end{align}
In particular, we have
\begin{align}\label{ga'}
\int_{0}^{T}\e^{-2\lambda s}\W_2(\Phi_s^{\L_{X_0}}(\mu),\Phi_s^{\L_{X_0}}(\nu))^2\d s\leq \frac{1}{2}\int_{0}^{T}\e^{-2\lambda s}\W_2(\mu_s,\nu_s)^2\d s.
\end{align}
Consider the space
$\tt E_{T}:= \{\mu\in C([0,T]; \scr P_2):\mu_0=\L_{X_0}\}$ equipped with the complete metric
$$
\tt \rr(\nu,\mu):=\left(\int_0^T\e^{-2\lambda t}\W_2(\nu_t,\mu_t)^2\d t\right)^{\frac{1}{2}}.
$$
\eqref{ga'} yields that $\Phi^{\L_{X_0}}$ is strictly contractive in $\tt E_T$, which together with the fixed point theorem implies that there exists a unique $\mu\in\tt E_T$ such that $\Phi_s^{\L_{X_0}}(\mu)=\mu_s,\ \ s\in[0,T]$. Thus, the strong well-posedness holds.
Moreover, if $X_t$ and $Y_t$ are two solutions to \eqref{E1}, \eqref{ga} holds for $\mu_s=\L_{X_s}$ and $\nu_s=\L_{Y_s}$. Again using $\W_2(\mu_s,\nu_s)^2\leq\E|X_s-Y_s|^2 $, we  deduce \eqref{X-Y}.
\end{proof}
%\subsection{Existence of Strong Solution}
%The next lemma characterizes the relationship between the existence of weak and strong solution
%(see \cite[Lemma 3.4]{HW}).
%\begin{lem}\label{SS} Let $(\bar\Omega, \{\bar\F_t\}_{t\geq 0},\bar\P)$ and $(\bar{X}_t,\bar{W}_t)$ be a weak solution to \eqref{E1} with $\mu_t:=\L_{\bar X_t}|_{\bar\P}$. If the SPDE
%\begin{align}\label{class}
%\d X(t)= B(t,X_t,\mu_t)\,\d t+b(t,X(t),\mu_t)\,\d t+ \sigma(t,X(t),\mu_t)\,\d W(t),\ \ 0\le t\le T
%\end{align}
%\begin{align}\label{class}
%\d X_t= \{A X_t+b_t(X_t,\mu_t)\}\d t+Q_t(X_t,\mu_t)\d W_t,\ \ 0\le t\le T
%\end{align}
%has a unique strong solution $X_t$ with $\L_{X_0}=\mu_0$, then   \eqref{E1} has a strong solution.
%\end{lem}
%\begin{proof} Since $\mu_t= \scr L_{\bar X_t}|_{\bar \P}$,   $\bar{X}_t$ is a weak solution to \eqref{class}. By Yamada-Watanabe principle, the strong uniqueness of \eqref{class} implies the weak uniqueness, so $\L_{X_t}=\mu_t, t\ge 0$. Therefore, $X_t$ is a strong solution to \eqref{E1}.
%\end{proof}
%\begin{rem}\label{wr} According to \cite{W}, \eqref{class} has a unique strong solution  under {\bf(a1)}-{\bf(a3)}. This together with Lemma \ref{SS} and Theorem \ref{ws} implies that \eqref{E1} has a strong solution.
%\end{rem}
\subsection{Weak Uniqueness under {\bf(a1)-(a3)}}
With the strong well-posedness in hand, it is routine to derive the weak uniqueness in $\scr P_2$.
\begin{thm}\label{wu} Assume {\bf(a1)-(a3)}. Then Equ. \eqref{E1} has weak uniqueness in $\scr P_2$.
\end{thm}
\begin{proof} Let   $(X_t)_{t\ge 0}$ solve \eqref{E1} with $\scr L_{X_0}=\mu_0\in\scr P_2$, and let   $(\tt X_t,\tt W_t)$ on
$(\tt\OO, \{\tt\F_t\}_{t\ge 0}, \tt\P)$ be a weak solution in $\scr P_2$ of \eqref{E1}  such that  $\L_{X_0}|_{\P}= \L_{\tt X_0}|_{\tt\P}=\mu_0$, i.e.
$\tt X_t$ solves
\beq\label{E1'0}
\d \tt X_t = A\tt X_t\d t+b_t(\tt X_t, \L_{\tt X_t}|_{\tt\P})\d t +Q_t(\tt X_t, \L_{\tt X_t}|_{\tt\P}) \d \tt W_{t},\ \ \ \scr L_{\tt X_0}=\mu_0.\end{equation}
 We intend to   prove   $\L_{X}|_{\P}=\L_{\tt X}|_{\tt\P}$.
  Let $\mu_t= \L_{X_t}|_{\P}$ and
$$\bar b_t(x)= b_t(x, \mu_t),\ \ \bar Q_t(x)= Q_t(x, \mu_t)\ \ x\in\H.$$
According to \cite[Theorem 1.1]{W}, the SPDE
\beq\label{E10} \d \bar X_t =A\bar{X}_t\d t+ \bar b_t(\bar X_t)\d t + \bar{Q}_t(\bar X_t) \d \tt W_{t}\,\ \ \bar X_0= \tt X_0 \end{equation}
has a unique mild solution under {\bf(a1)}-{\bf(a3)}.
By \cite[Theorem 2]{O}, it also satisfies weak uniqueness. Noting that
$$\d X_t= AX_t\d t+\bar b_t(X_t)\d t + \bar{Q}_t(X_t) \d W_{t},\ \ \L_{X_0}|_{\P}= \L_{\tt X_0}|_{\tt\P},$$ the weak uniqueness of \eqref{E10} implies
\beq\label{HW} \L_{\bar X}|_{\tt\P}= \L_X|_{\P}.\end{equation}
Hence, \eqref{E10} can be rewritten as
$$ \d \bar X_t = A\bar X_t\d t+b_t(\bar X_t, \L_{\bar X_t}|_{\tt\P})\d t + Q_t(\bar X_t, \L_{\bar X_t}|_{\tt\P})\d \tt W_{t},\ \ \bar X_0=\tt X_0.$$
By the strong well-posedness in $\scr P_2$ of Equ. \eqref{E1} according to Theorem \ref{T2.1}(1),
we obtain  $\bar X=\tt X$. Therefore,  \eqref{HW} implies $\L_{\tt X}|_{\tt \P} = \L_X|_{\P}$.
\end{proof}
\begin{rem}\label{wek} If $Q_t(x,\mu)$ does not depend on $\mu$, the weak uniqueness can be ensured in the case that $b$ is not weakly continuous in the distribution variable, see \cite[Theorem 1.1(1)]{HW20} and references therein for the condition that $b$ is Lipschitz continuous in distribution variable under total variational distance. The crucial technique is Girsanov's transform, which is also available in infinite dimensional situation.
\end{rem}
\begin{proof}[Proof of Theorem \ref{T2.1}(1)] Since the strong solution is also a weak solution, the weak well-posedness can be obtained by Theorem \ref{T2.1}(2) and Theorem \ref{wu}. Moreover, \eqref{X-Y} implies \eqref{Pta}.
\end{proof}
\section{Proof of Theorem \ref{THar} and Theorem \ref{TsHar}}
The main idea of the proof is to fix the distribution in the coefficients of Equ. \eqref{E1}, which goes back to the classical situation. Then the log-Harnack inequality from different initial distribution holds according to \cite[(1.7)]{W}. Next, we calculate the relative entropy for two solutions with different distributions in the coefficients of \eqref{E1} but same initial distribution, which implies the total variational distance of these two solutions by Pinsker's inequality. Combining the above two parts, the desired log-Harnack inequality follows.  As for the Harnack inequality with power and shift Harnack inequality, the coupling by change of measure is used.

\subsection{Proof of Theorem \ref{THar}}
\begin{proof} (1) According to \cite[Theorem 1.4.2(2)]{Wbook}, \eqref{pke} follows from log-Harnack inequality and Pinsker's inequality. \eqref{ap1} is a direct conclusion of Harnack inequality with power, see \cite[Theorem 1.4.2(1)]{Wbook}. So we only need to prove log-Harnack inequality and Harnack inequality with power.

Let $\mu_t=P_t^\ast \mu_0$ and $\nu_t=P_t^\ast \nu_0$. Let $X_t$ be the solution to SPDEs
\beq\label{EC0} \d X_t= AX_t\d t+b_t(X_t,\mu_t)\d t+Q_t(X_t) \d W_t\end{equation} with
$\L_{X_0}=\mu_0$.
Define
$$
\gamma_s=Q_s^\ast(Q_sQ_s^\ast)^{-1}(X_s)[b_s(X_s,\mu_s)-b_s(X_s,\nu_s)],
\ \ \ \ \
\bar{W}_t=W_t+\int_0^t\gamma_s\d s,
$$
and
$$R_T=\exp\left\{-\int_0^T\<\gamma_s,\d W_s\>-\frac{1}{2}\int_0^T|\gamma_s|^2\d s\right\}. $$
By {\bf(a2)}-{\bf(a3)} and \eqref{Pta}, Girsanov's theorem yields that $\{\bar{W_s}\}_{s\in[0,T]}$ is a cylindrical Brownian motion under $\Q_T=R_T\P$.
Moreover, from \eqref{1.2}, {\bf(a2)} and \eqref{Pta}, it is  clear that
\begin{align}\label{RT}
\log\E R_T^2&=\log\E\exp\left\{-\int_0^T2\<\gamma_s,\d W_s\>-\int_0^T|\gamma_s|^2\d s\right\}\\ \nonumber
&\leq C(T)\int_{0}^T\W_2(\mu_s,\nu_s)^2\d s\leq C(T)\W_2(\mu_0,\nu_0)^2.
\end{align}
for some constant $C(T)>0$.
Then we have
\beq\label{ECb}
\d X_t= AX_t\d t+b_t(X_t,\nu_t)\d t+Q_t(X_t) \d \bar{W}_t.
\end{equation}
Letting $\bar{\mu}_t$ be the distribution of $X_t$ under $\Q_T$, we derive
\begin{align}
\label{mub}\bar{\mu}_T(f)=\E^{\Q_T}f(X_T)=\E(R_Tf(X_T))=\E(\E(R_T|X_T)f(X_T)),\ \ f\in\B_b(\H).
\end{align}
This implies $\P$-a.s.
\begin{align}\label{muc}\frac{\d\bar{\mu}_T}{\d\mu_T}(X_T)=\E(R_T|X_T).
\end{align}
%Next, consider
%\beq\label{EC00} \d Y_t= AY_t+b(Y_t,\nu_t)+Q_t(Y_t) \d \bar{W}_t\end{equation} with
%$\L_{Y_0}=\nu_0$.
On the other hand, according to the log-Harnack inequality in \cite[(1.7)]{W} and \cite[Theorem 1.4.2(2)]{Wbook}, there exists a constant $C>0$ such that
$$
\mathrm{Ent}(P_T^\ast\nu_0|\bar{\mu}_T)
=\bar{\mu}_T\left(\frac{\d P_T^\ast\nu_0}{\d \bar{\mu}_T}
\log\frac{\d P_T^\ast\nu_0}{\d \bar{\mu}_T}\right)\leq \frac{C}{T\wedge1}\W_2(\mu_0,\nu_0)^2.
$$
Thus, by Young's inequality, Jensen's inequality, \eqref{RT}, \eqref{mub} and \eqref{muc}, for any $f\in\B_b(\H)$,
one can arrive at
\begin{align*}
&P_T\log f(\nu_0)\\
=&\mu_T\left(\frac{\d \bar{\mu}_T}{\d \mu_T}\frac{\d P_T^\ast\nu_0}{\d \bar{\mu}_T}\log f\right)\\
\leq&\log P_Tf(\mu_0)+\mu_T\left(\frac{\d \bar{\mu}_T}{\d \mu_T}\frac{\d P_T^\ast\nu_0}{\d \bar{\mu}_T}\log\left(\frac{\d \bar{\mu}_T}{\d \mu_T}\frac{\d P_T^\ast\nu_0}{\d \bar{\mu}_T}\right)\right)\\
=&\log P_Tf(\mu_0)+\mu_T\left(\frac{\d \bar{\mu}_T}{\d \mu_T}\frac{\d P_T^\ast\nu_0}{\d \bar{\mu}_T}\log\frac{\d \bar{\mu}_T}{\d \mu_T}\right)
+\mu_T\left(\frac{\d \bar{\mu}_T}{\d \mu_T}\frac{\d P_T^\ast\nu_0}{\d \bar{\mu}_T}
\log\frac{\d P_T^\ast\nu_0}{\d \bar{\mu}_T}\right)\\
=&\log P_Tf(\mu_0)+\bar{\mu}_T\left(\frac{\d P_T^\ast\nu_0}{\d \bar{\mu}_T}
\log\frac{\d \bar{\mu}_T}{\d \mu_T}\right)
+\bar{\mu}_T\left(\frac{\d P_T^\ast\nu_0}{\d \bar{\mu}_T}
\log\frac{\d P_T^\ast\nu_0}{\d \bar{\mu}_T}\right)\\
\leq& \log P_Tf(\mu_0)+\log \bar{\mu}_T\left(\frac{\d \bar{\mu}_T}{\d \mu_T}\right)
+2\bar{\mu}_T\left(\frac{\d P_T^\ast\nu_0}{\d \bar{\mu}_T}
\log\frac{\d P_T^\ast\nu_0}{\d \bar{\mu}_T}\right)\\
\leq&\log P_Tf(\mu_0)+\log \E R_T^2
+2\bar{\mu}_T\left(\frac{\d P_T^\ast\nu_0}{\d \bar{\mu}_T}
\log\frac{\d P_T^\ast\nu_0}{\d \bar{\mu}_T}\right)\\
\leq&\log P_Tf(\mu_0)+C(T)\W_2(\mu_0,\nu_0)^2+\frac{C}{T\wedge1}\W_2(\mu_0,\nu_0)^2\\
\leq&\log P_Tf(\mu_0)+\frac{C(T)}{T\wedge1}\W_2(\mu_0,\nu_0)^2
\end{align*}
for some constant $C(T)>0$.

(2) Recall $\mu_t=P_t^\ast \mu_0$ and $\nu_t=P_t^\ast \nu_0$. Let $X_t, Y_t$ solve the equations respectively
\beq\label{EC1} \begin{split}
&\d X_t= AX_t\d t+b_t(X_t,\mu_t)\d t+Q_t \d W_t,\\
&\d Y_t= AY_t\d t+b_t(X_t,\mu_t)\d t+Q_t \d W_t+ \e^{A t}\frac{X_0-Y_0}{T}\d t
\end{split}\end{equation}
with $\L_{X_0}=\mu_0$ and
$\L_{Y_0}=\nu_0$. Then we have $Y_t=X_t+\e^{A t}\frac{(T-t)(Y_0-X_0)}{T}$.
In particular, $Y_T=X_T$.
Let
\begin{align*}
\tilde{\Phi}(t)&=b_t(X_t,\mu_t)-b_t(Y_t,\nu_t)+\e^{A t}\frac{X_0-Y_0}{T}, \ \ t\in[0,T],
\end{align*}
and $$M_s=\int_0^s\< Q_u^\ast(Q_uQ_u^\ast)^{-1}\tilde{\Phi}(u), \d
W_u\>,\ \ s\in[0,T].$$
Set
\begin{align*}
\tilde{R}(s)=\exp\left(-M_s-\frac{1}{2}\<M\>_s\right),\ \ s\in[0,T],
\end{align*}
and
$$
\tilde{W}_s=W_s+\int_0^sQ_u^\ast(Q_uQ_u^\ast)^{-1}\tilde{\Phi}(u)\d u,\ \ s\in[0,T].
$$
In addition, combining {\bf(a3)} with \eqref{Pta}, there exists a constant $C>0$ such that for any $t\in[0,T]$,
\begin{align*}
\int_0^T|\tilde{\Phi}(t)|^2\d t&\leq \int_0^T\left\{2|b_t(X_t,\mu_t)-b_t(Y_t,\nu_t)|^2+2\left|\e^{A t}\frac{X_0-Y_0}{T}\right|^2\right\}\d t\\
&\leq \int_{0}^T4\phi^2\left(\frac{T-t}{T}|X_0-Y_0|\right)\d t+\int_{0}^T4K(T)^2\W_2(\mu_t,\nu_t)^2\d t+2\frac{|X_0-Y_0|^2}{T}\\
&\leq 4T\phi^2\left(|X_0-Y_0|\right)+C(T)\W_2(\mu_0,\nu_0)^2+2\frac{|X_0-Y_0|^2}{T}.
\end{align*}
By Girsanov's theorem, $\{\tilde{W_s}\}_{s\in[0,T]}$ is a cylindrical Brownian motion under $\tilde{\Q}=\tilde{R}(T)\P$.
Then the second equation in (\ref{EC1}) can be rewritten as
\beq\label{E29}
\d Y_t= AY_t\d t+b_t(Y_t,\nu_t)\d t+Q_t\d \tilde{W}_t.
\end{equation}
Consider SPDEs
\beq\label{E2'}
\d \tilde{Y}_t=A\tilde{Y}_t\d t+b_t(\tilde{Y}_t,\L_{\tilde{Y}_t}|_{\tilde{\Q}})\d t
+Q_t\d \tilde{W}_t
\end{equation}
with $\tilde{Y}_0=Y_0$, then $\L_{Y_0}|_\P=\L_{Y_0}|_{\tilde{\Q}}=\L_{\tilde{Y}_0}|_{\tilde{\Q}}=\nu_0$.
Thus, by the weak uniqueness, $\L_{\tilde{Y}_t}|_{\tilde{\Q}}=\nu_t$, which implies $\tilde{Y}_t=Y_t$ and $\L_{Y_t}|_{\tilde{\Q}}=\nu_t$.

On the other hand,
%by Young's inequality, we obtain
%\begin{align*}
%P_T \log f(\nu_0)&=\E^{\tilde{\Q}}\log f(Y_T)\\
%&=\E ^{\tilde{\Q}}\log f(X_T)\leq \log P_T f(\mu_0)+\E \tilde{R}(T)\log \tilde{R}(T),
%\end{align*}
by H\"{o}lder's inequality, for any $p>1$, it holds
\begin{align*}
P_T f(\nu_0)=\E^{\tilde{\Q}}f(Y_T)
&=\E ^{\tilde{\Q}}f(X_T)\leq (P_T f^p(\mu_0))^{\frac{1}{p}}\{\E \tilde{R}(T)^{\frac{p}{p-1}}\}^{\frac{p-1}{p}}.
\end{align*}
By the definition of $\tilde{R}(T)$ and {\bf (a2)}, one can obtain
%\begin{align*}
%\E \tilde{R}(T)\log \tilde{R}(T)=\E ^{\tilde{\Q}}\log \tilde{R}(T)=\frac{1}{2}\E^{\tilde{\Q}} \int_0^T |(Q_sQ_s^\ast)^{-1}\tilde{\Phi}(u)|^2\d u,
%\end{align*}
%and
\begin{align*}
&\E \tilde{R}(T)^{\frac{p}{p-1}}\\
&\leq \E\Bigg\{\exp\bigg[-\frac{p}{p-1}M_T-\frac{1}{2}\frac{p^2}{(p-1)^2}\<M\>_T\bigg] \times\exp\bigg[\frac{1}{2}\frac{p^2}{(p-1)^2}-\frac{1}{2}\frac{p}{p-1}\<M\>_T\bigg]\Bigg\}\\
&\leq \E\Bigg\{\E\Bigg\{\exp\bigg[-\frac{p}{p-1}M_T-\frac{1}{2}\frac{p^2}{(p-1)^2}\<M\>_T\bigg] \bigg|\F_0\Bigg\}\\
& \ \ \ \times\exp\left\{\frac{p}{2(p-1)^2}K(T)\left(4T\phi^2 \left(|X_0-Y_0|\right)+C(T)\W_2(\mu_0,\nu_0)^2+2\frac{|X_0-Y_0|^2}{T}\right)\right\}\Bigg\}\\
&\leq\E\exp\left\{\frac{p}{2(p-1)^2}K(T) \left(4T\phi^2\left(|X_0-Y_0|\right)+C(T)\W_2(\mu_0,\nu_0)^2+2\frac{|X_0-Y_0|^2}{T}\right)\right\}.
\end{align*}
Thus, we derive the Harnack inequalities.
\end{proof}
\subsection{Proof of Theorem \ref{TsHar}}
\begin{proof} Recall $\mu_t=P_t^\ast\mu_0$.
Let $X_t, Y_t$ solve the equations
\beq\label{EC1s}\begin{split}
&\d X_t= AX_t\d t+b_t(X_t,\mu_t)\d t+Q_t(\mu_t) \d W_t, \ \ \L_{X_0}=\mu_0,\\
&\d Y_t= AY_t\d t+b_t(X_t,\mu_t)\d t+Q_t(\mu_t) \d W_t+ \e^{A t}\frac{y}{T}\d t, \ \ Y_0=X_0.
\end{split}\end{equation}
Then we have $Y_t=X_t+\e^{A t}\frac{ty}{T}$.
In particular, $Y_T=X_T+\e^{AT}y$.
Let
\begin{align*}
\bar{\Phi}(t)&=b_t(X_t,\mu_t)-b_t(Y_t,\mu_t)+\e^{A t}\frac{y}{T},\ \ t\in[0,T].
\end{align*}
For any $t\in[0,T]$, set
\begin{align*}
\bar{R}(t)&=\exp\bigg[-\int_0^t\< (Q_u^\ast(Q_uQ_u^\ast)^{-1})(\mu_u)\bar{\Phi}(u), \d
W_u\>-\frac{1}{2}\int_0^t |(Q_u^\ast(Q_uQ_u^\ast)^{-1})(\mu_u)\bar{\Phi}(u)|^2\d u\bigg],
\end{align*}
and
$$
\bar W_t=W_t+\int_0^t(Q_u^\ast(Q_uQ_u^\ast)^{-1})(\mu_u)\bar{\Phi}(u)\d u.
$$
There exists a constant $C>0$ such that for any $t\in[0,T]$,
\beg{equation}\label{NN0s}
\beg{split}
|\bar{\Phi}(t)|
\leq \phi\left(\left|\e^{A t}\frac{ty}{T}\right|\right)+\left|\e^{A t}\frac{y}{T}\right|.
\end{split}\end{equation}
Thus, we have
\begin{equation}
\begin{split}\label{Phis}
\int_0^T|\bar{\Phi}(s)|^2\d s&\leq 2T\phi^2(|y|)+2\frac{|y|^2}{T}.
\end{split}\end{equation}
Girsanov's theorem implies that $\{\bar{W_s}\}_{s\in[0,T]}$ is a cylindrical Brownian motion under $\bar{\Q}_T=\bar{R}(T)\P$.
Then the second equation in (\ref{EC1s}) can be reformulated as
\beq\label{E2s}
\d Y_t= AY_t\d t+b_t(Y_t,\mu_t)\d t+Q_t(\mu_t) \d \bar{W}_t, \ \ Y_0=X_0.
\end{equation}
Thus, the distribution of $Y_T$ under the new probability $\bar{\Q}_T$ coincides with the one of $X_T$  under $\P$.

On the other hand, by Young's inequality and H\"{o}lder's inequality respectively, we arrive at
\begin{align*}
P_T \log f(\mu_0)&=\E^{\bar{\Q}_T}\log f(Y_T)\\
&=\E ^{\bar{\Q}_T}\log f(X_T+\e^{AT} y)\\
&\leq \log P_T f(\cdot+\e^{AT} y)(\mu_0)+\E \bar{R}(T)\log \bar{R}(T),
\end{align*}
and
\begin{align*}
P_T f(\mu_0)&=\E^{\bar{\Q}_T}f(Y_T)\\
&=\E ^{\bar{\Q}_T}f(X_T+\e^{AT}y)\leq (P_T f^p(\cdot+\e^{AT}y))^{\frac{1}{p}}(\mu_0)\{\E \bar{R}(T)^{\frac{p}{p-1}}\}^{\frac{p-1}{p}}.
\end{align*}
It is standard to obtain
\begin{align*}
\E \bar{R}(T)\log \bar{R}(T)=\E ^{\bar{\Q}_T}\log \bar{R}(T)=\frac{1}{2}\E^{\bar{\Q}_T} \int_0^T |(Q_u^\ast(Q_uQ_u^\ast)^{-1})(\mu_u)\bar{\Phi}(u)|^2\d u,
\end{align*}
and by the same argument as in the estimate of $\E\tilde{R}(T)^{\frac{p}{p-1}}$ in Section 4.1, it holds
\begin{align*}
\E \bar{R}(T)^{\frac{p}{p-1}}&\leq\mathrm{ess}\sup_{\Omega}\exp\left\{\frac{p}{2(p-1)^2}\int_0^T |(Q_u^\ast(Q_uQ_u^\ast)^{-1})(\mu_u)\bar{\Phi}(u)|^2\d u\right\}.
\end{align*}
Thus, the shift Harnack inequality follows from \eqref{Phis} and {\bf(a2)}.
\end{proof}

\beg{thebibliography}{99}

\bibitem{BR1} V. Barbu, M. R\"ockner, \emph{Probabilistic representation for solutions to non-linear Fokker-Planck equations,} SIAM J. Math. Anal. 50(2018), 4246-4260.

\bibitem{BR2} V. Barbu, M. R\"ockner, \emph{From non-linear Fokker-Planck equations to solutions of distribution dependent SDE,} arXiv:1808.10706.

\bibitem{BB0}M. Bauer, T. M.-Brandis, \emph{McKean-Vlasov equations on infinite-dimensional Hilbert spaces with irregular drift and additive fractional noise,} arXiv:1912.07427.

\bibitem{BB} M. Bauer, T. M-Brandis, \emph{Existence and Regularity of Solutions to Multi-Dimensional Mean-Field Stochastic Differential Equations with Irregular Drift,}	 arXiv:1912.05932.

\bibitem{BBP} M. Bauer, T. M-Brandis, F. Proske,\emph{Strong Solutions of Mean-Field Stochastic Differential Equations with irregular drift,} arXiv:1806.11451.

\bibitem{BG} Z. Brze\'{z}niak, D. Gatarek, \emph{Martingale solutions and invariant measures for stochastic evolution equations in Banach spaces,} Stoch. Proc. Appl. 84(1999), 187-225.

\bibitem{CR} P. E. Chaudru de Raynal, \emph{Strong well-posedness of McKean-Vlasov stochastic differential equation with H\"older drift, } DOI: 10.1016/j.spa.2019.01.006.

\bibitem{CMF} M. F. Chen, \emph{From Markov chain to non-equilibrium particle systems}(second edition), World Scientific, 2004.

\bibitem{CF} L. Campi, M. Fischer, \emph{$N$-player games and mean-field games with absorption,} Ann. Appl. Probab. 28(2016), 2188-2242.

\bibitem{DZ} G. Da Prato, J. Zabczyk,  \emph{Stochastic Equations in Infinite Dimensions,} Encyclopedia of Mathematics and its Applications, Cambridge University Press, 1992.

\bibitem{FS}S. Feng, \emph{Large deviation for empirical process of interacting particle system
    with unbounded jumps}, Ann. Prob. 22(1994), 2122--2151.

\bibitem{HRW} X. Huang, M. R\"{o}ckner, F.-Y. Wang, \emph{Nonlinear Fokker--Planck equations for probability measures on path space and path-distribution dependent SDEs,} Discrete Contin. Dyn. Syst. 39(2019), 3017-3035.

\bibitem{HW} X. Huang, F.-Y.  Wang,  \emph{Distribution dependent SDEs with singular coefficients,} Stoch. Proc. Appl. 129(2019), 4747-4770.

\bibitem{HW20} X. Huang, F.-Y.  Wang,  \emph{McKean-Vlasov   SDEs with Drifts  Discontinuous under  Wasserstein Distance,} arXiv:2002.06877.

\bibitem{L} W. Liu, Harnack \emph{inequality and applications for stochastic evolution equations with monotone drifts}, J. Evol. Equ. 9(2009), 747-770.

\bibitem{M} H. P. McKean, \emph{A class of Markov processes associated with nonlinear parabolic equations,} Proc Natl Acad Sci U S A 56(1966), 1907-1911.

\bibitem{MV} Yu. S. Mishura, A. Yu. Veretennikov, \emph{Existence and uniqueness theorems for solutions of McKean-Vlasov stochastic equations,} arXiv:1603.02212.

\bibitem{O} M. Ondrejet,  \emph{Uniqueness for stochastic evolution equations in Banach spaces,} Dissertationes Math. (Rozprawy Mat.) 426(2004).

\bibitem{Pin} M. S. Pinsker, \emph{ Information and Information Stability of Random Variables and Processes,} Holden-Day, San Francisco, 1964.

\bibitem{RW10} M. R\"ockner, F.-Y. Wang, \emph{Harnack and functional inequalities for generalized Mehler semigroups,}  J. Funct. Anal.  203(2007), 237-261.

\bibitem{RZ} M. R\"ockner, X. Zhang, \emph{Well-posedness of distribution dependent SDEs with singular drifts,} arXiv:1809.02216.

\bibitem{RW} M. R\"ockner, F.-Y Wang, \emph{Log-Harnack inequality for stochastic differential equations in Hilbert spaces and its consequences}, Infin. Dimenns. Anal. Quantum Probab. Relat. Top. 13(2010), 27-37.

\bibitem{V} A. A. Vlasov, \emph{The vibrational properties of an electron gas,} Soviet Physics Uspekhi 10(1968), 721.

\bibitem{W97} F.-Y. Wang, \emph{ On estimation of the logarithmic Sobolev constant and gradient
estimates of heat semigroups}. Probab. Theory Relat. Fields. 108(1997), 87-101.

\bibitem{W11} F.-Y.  Wang, \emph{ Harnack inequality for SDE with multiplicative noise and
extension to Neumann semigroup on nonconvex manifolds,} Ann. Probab. 39(2011),  1449-1467.

 \bibitem{W14a} F.-Y. Wang, \emph{Integration by parts formula and shift Harnack inequality for stochastic  equations,}   Ann. Probab. 42(2014), 994-1019.

\bibitem{Wbook} F.-Y. Wang, \emph{Harnack Inequality and Applications for Stochastic Partial Differential Equations,} Springer, New York, 2013.

\bibitem{W07} F.-Y. Wang, \emph{Harnack inequality and applications for stochastic generalized porous media equations,}  Ann. Probab. 35(2007), 1333-1350.

\bibitem{W} F.-Y. Wang,  \emph{Gradient estimate and applications for SDEs in Hilbert space with multiplicative noise and Dini continuous drift,}  J. Differential Equations 260(2016), 2792-2829.

\bibitem{FYW1} F.-Y. Wang, \emph{Distribution-dependent SDEs for Landau type equations,} Stoch. Proc. Appl. 128(2018), 595-621.

\bibitem{WZ} F.-Y. Wang, T. Zhang,  \emph{Log-Harnack inequalities for semilinear SPDE with strongly multiplicative noise,}  Stoch. Proc. Appl. 124(2014), 1261-1274.

\bibitem{AZ} A. K. Zvonkin,  \emph{A transformation of the phase space of a diffusion process that removes the drift,}  Math. Sb. 93(1974), 129-149.
\end{thebibliography}

\end{document}